\documentclass[11pt,reqno]{amsart}

\usepackage{iftex}
\ifPDFTeX
  \usepackage{cmap}
  \makeatletter\expandafter\let\csname cmap@set@OMX\endcsname\empty\makeatother
  \input{glyphtounicode}
  \newcommand{\gtu}[2]{\pdfglyphtounicode{#1big}{#2}%
    \pdfglyphtounicode{#1Big}{#2}\pdfglyphtounicode{#1bigg}{#2}%
    \pdfglyphtounicode{#1Bigg}{#2}}
  \newcommand{\gtuop}[2]{\pdfglyphtounicode{#1text}{#2}%
    \pdfglyphtounicode{#1display}{#2}}
  \gtu{parenleft}{0028}\gtu{parenright}{0029}\gtu{bracketleft}{005B}
  \gtu{bracketright}{005D}\gtu{braceleft}{007B}\gtu{braceright}{007D}
  \gtu{floorleft}{230A}\gtu{floorright}{230B}\gtu{ceilingleft}{2308}
  \gtu{ceilingright}{2309}\gtu{angbracketleft}{27E8}\gtu{angbracketright}{27E9}
  \gtu{radical}{221A}\gtu{slash}{002F}\gtu{backslash}{005C}
  \pdfglyphtounicode{vextendsingle}{007C}
  \pdfglyphtounicode{vextenddouble}{2016}
  \gtuop{summation}{2211}\gtuop{product}{220F}\gtuop{coproduct}{2210}
  \gtuop{union}{22C3}\gtuop{intersection}{22C2}\gtuop{unionmulti}{228E}
  \gtuop{unionsq}{2294}\gtuop{logicaland}{22C0}\gtuop{logicalor}{22C1}
  \gtuop{integral}{222B}\gtuop{contintegral}{222E}\gtuop{circledot}{2A00}
  \gtuop{circleplus}{2A01}\gtuop{circlemultiply}{2A02}
\fi
\usepackage[T1]{fontenc}
\usepackage{amssymb,mathtools}
\usepackage{enumitem}
\usepackage{tikz,pgfplots}
\pgfplotsset{compat=1.18}
\usetikzlibrary{arrows.meta,positioning,patterns,calc,decorations.pathreplacing,
  shapes.geometric,fit}
\usepackage{hyperref}
\hypersetup{colorlinks=true,linkcolor=blue!60!black,citecolor=red!60!black,
  urlcolor=blue!60!black,
  pdftitle={Resolving Erdos-Ulam Monochromatic Union-Closed Family Conjectures},
  pdfauthor={Deep Bhattacharjee, Priyabrata Mandal, Ushashi Bhattacharya}}

\numberwithin{equation}{section}

\theoremstyle{plain}
\newtheorem{theorem}{Theorem}[section]
\newtheorem{lemma}[theorem]{Lemma}
\newtheorem{proposition}[theorem]{Proposition}
\newtheorem{corollary}[theorem]{Corollary}
\newtheorem{question}[theorem]{Question}
\theoremstyle{definition}
\newtheorem{definition}[theorem]{Definition}
\theoremstyle{remark}
\newtheorem{remark}[theorem]{Remark}

\newcommand{\cF}{\mathcal{F}}
\newcommand{\cL}{\mathcal{L}}
\newcommand{\cP}{\mathcal{P}}
\newcommand{\cS}{\mathcal{S}}
\newcommand{\cY}{\mathcal{Y}}
\newcommand{\Ex}{\mathbb{E}}
\newcommand{\eps}{\varepsilon}
\DeclareMathOperator{\vc}{vc}
\newcommand{\zenododoi}{10.5281/zenodo.23094057}

\usepackage{etoolbox}
\makeatletter
\patchcmd{\@setaddresses}{\footnotesize}{\footnotesize\raggedright}{}{}
\makeatother

\title[Erd\H{o}s--Ulam union-closed family conjectures]{Resolving
Erd\H{o}s--Ulam Monochromatic Union-Closed Family Conjectures}

\author[D. Bhattacharjee]{Deep Bhattacharjee$^{\dagger}$}
\address[Deep Bhattacharjee]{Formerly, Electro-Gravitational Space Propulsion Laboratory (EGSPL),
  Bhubaneswar, Odisha 751030, India}
\email{d.bhattacharjee@erl-forschung.de}
\email{itsdeep@live.com}
\urladdr{\href{https://orcid.org/0000-0003-0466-750X}{0000-0003-0466-750X}}

\author[P. Mandal]{Priyabrata Mandal}
\address[Priyabrata Mandal]{Department of Mathematics, Bioinformatics and Computer
  Applications, Maulana Azad National Institute of Technology Bhopal,
  Bhopal 462003, India}
\email{priyabrata@manit.ac.in}
\urladdr{\href{https://orcid.org/0000-0001-6472-6239}{0000-0001-6472-6239}}

\author[U. Bhattacharya]{Ushashi Bhattacharya}
\address[Ushashi Bhattacharya]{Formerly, National Taiwan University, Taipei 106319, Taiwan}
\email{rai.bhattacharya.in@gmail.com}
\urladdr{\href{https://orcid.org/0009-0002-2254-3914}{0009-0002-2254-3914}}

\thanks{$^{\dagger}$\,Corresponding author: Deep Bhattacharjee.}

\subjclass[2020]{Primary 11B75, 05D10; Secondary 05D05, 11B25, 68V20}
\keywords{union-closed families, Ramsey theory, Hilbert cubes, box norm,
van der Waerden's theorem, Sauer--Shelah lemma, Lean}
\date{}

\begin{document}
\sloppy

\begin{abstract}
We prove both Erd\H{o}s--Ulam conjectures on monochromatic union-closed
families. Every two-colouring of the subsets of a finite set contains a
monochromatic union-closed family whose size grows faster than any fixed
power of the size of the set, while suitable colourings admit no
monochromatic union-closed family of exponential size. Both results hold for any number of colours
and are verified in Lean.
\end{abstract}

\maketitle

\section{Introduction}
\label{sec:intro}

Write $[n]=\{1,\dots,n\}$ and $2^{[n]}$ for the family of all subsets of
$[n]$. For an integer $k\ge2$, a \emph{$k$-colouring} of $2^{[n]}$ is a
map $\chi:2^{[n]}\to[k]$, and a family $\cF\subseteq2^{[n]}$ is
\emph{monochromatic} if $\chi$ is constant on $\cF$. The family $\cF$ is
\emph{union-closed} if $A\cup B\in\cF$ whenever $A,B\in\cF$, and it is a
\emph{sublattice} if it is closed under both union and intersection.
Erd\H{o}s and Ulam asked how large a monochromatic family of either kind
a colouring is forced to contain. Accordingly let
\begin{align*}
f(n)&=\min_{\chi}\,\max\bigl\{|\cL|:\ \cL\subseteq2^{[n]}
  \text{ is a monochromatic sublattice}\bigr\},\\
F_k(n)&=\min_{\chi}\,\max\bigl\{|\cF|:\ \cF\subseteq2^{[n]}
  \text{ is a monochromatic union-closed family}\bigr\},
\end{align*}
where the first minimum is over all $2$-colourings and the second over
all $k$-colourings of $2^{[n]}$, and write $F=F_2$. A single set is a
sublattice and every sublattice is union-closed, so $1\le f(n)\le F(n)$.
A maximal chain $\emptyset\subset\{1\}\subset\{1,2\}\subset\dots\subset[n]$
has $n+1$ members, at least $\lceil(n+1)/2\rceil$ of them receive the same
colour under a $2$-colouring, and every subchain is a sublattice. Hence
\begin{equation}
\label{eq:trivial}
F(n)\ \ge\ f(n)\ \ge\ \Bigl\lceil\frac{n+1}{2}\Bigr\rceil .
\end{equation}
A $2$-colouring is also a $k$-colouring, so $F_k(n)\le F(n)$ for every
$k\ge2$. Figure~\ref{fig:lattice} shows the smallest interesting case.

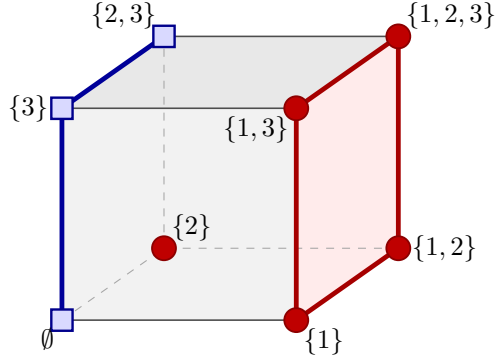
\begin{figure}[t]
\centering
\begin{tikzpicture}[
    x={(3.1cm,0cm)}, y={(1.35cm,0.95cm)}, z={(0cm,2.8cm)},
    red/.style={circle, draw=red!55!black, fill=red!75!black, line width=0.7pt,
      inner sep=0pt, minimum size=9pt},
    blue/.style={rectangle, draw=blue!55!black, fill=blue!15, line width=0.9pt,
      inner sep=0pt, minimum size=8pt},
    lab/.style={font=\small, inner sep=2pt}]
  \fill[gray!10] (0,0,0) -- (1,0,0) -- (1,0,1) -- (0,0,1) -- cycle;
  \fill[gray!18] (0,0,1) -- (1,0,1) -- (1,1,1) -- (0,1,1) -- cycle;
  \fill[red!8]   (1,0,0) -- (1,1,0) -- (1,1,1) -- (1,0,1) -- cycle;
  \draw[dashed, gray!70] (0,0,0) -- (0,1,0) -- (1,1,0);
  \draw[dashed, gray!70] (0,1,0) -- (0,1,1);
  \draw[gray!60!black, line width=0.6pt]
    (0,0,0) -- (1,0,0) -- (1,1,0) -- (1,1,1) -- (0,1,1) -- (0,0,1) -- cycle
    (1,0,0) -- (1,0,1) -- (1,1,1)
    (0,0,1) -- (1,0,1);
  \draw[red!65!black, line width=1.8pt]
    (1,0,0) -- (1,1,0) -- (1,1,1) -- (1,0,1) -- cycle;
  \draw[blue!60!black, line width=1.8pt] (0,0,0) -- (0,0,1) -- (0,1,1);
  \node[blue] (e)   at (0,0,0) {};
  \node[red]  (a)   at (1,0,0) {};
  \node[red]  (b)   at (0,1,0) {};
  \node[blue] (c)   at (0,0,1) {};
  \node[red]  (ab)  at (1,1,0) {};
  \node[red]  (ac)  at (1,0,1) {};
  \node[blue] (bc)  at (0,1,1) {};
  \node[red]  (abc) at (1,1,1) {};
  \node[lab, below left=1pt]            at (e)   {$\emptyset$};
  \node[lab, below right=1pt]           at (a)   {$\{1\}$};
  \node[lab, above right=1pt]           at (b)   {$\{2\}$};
  \node[lab, left=3pt]                  at (c)   {$\{3\}$};
  \node[lab, right=3pt]                 at (ab)  {$\{1,2\}$};
  \node[lab, below left=1pt]            at (ac)  {$\{1,3\}$};
  \node[lab, above left=1pt]            at (bc)  {$\{2,3\}$};
  \node[lab, above right=1pt]           at (abc) {$\{1,2,3\}$};
\end{tikzpicture}
\caption{A $2$-colouring of the subsets of $\{1,2,3\}$, drawn on the
cube $2^{[3]}$. The five red sets (discs) form a union-closed family which
is not a sublattice, since $\{1\}\cap\{2\}=\emptyset$ is blue. The three
blue sets (squares) form a chain.}
\label{fig:lattice}
\end{figure}

The question is Problem~1183 in Bloom's collection of Erd\H{o}s problems
\cite{Bl26}, where it is listed as open. As recorded there, Erd\H{o}s had
no plausible conjecture for the true order of magnitude of $F(n)$, but
asked whether
\begin{equation}
\label{eq:conj}
F(n)\ \ge\ n^{\omega(n)}\ \text{ for some }\omega(n)\to\infty,
\qquad\text{and}\qquad
F(n)\ <\ (1+o(1))^{n}.
\end{equation}
Erd\H{o}s also noted that Howorka had proved the first statement for
colourings in which the colour of a set depends only on its
size.\footnote{The problem page states Howorka's result without a
reference, and we have not found Howorka's argument in print.
Section~\ref{sec:card} gives a complete proof, with an explicit bound in
terms of van der Waerden numbers.} We prove both statements in
\eqref{eq:conj} for arbitrary colourings, and for any number of colours.

\subsection{Results}

The main result is a polynomial lower bound of every degree, with an
explicit constant.

\begin{theorem}
\label{thm:lower}
Let $k\ge2$ and $t\ge1$ be integers, put $M=t^{2}k^{2^{t}}$ and
\[
c_{k,t}=\Bigl(k^{2^{t}+1}\,t^{t}\,(4M^{2})^{t+1}\Bigr)^{-1}.
\]
Every $k$-colouring of $2^{[n]}$ contains a monochromatic union-closed
family with at least $c_{k,t}\,n^{t+1}$ members; that is,
$F_k(n)\ge c_{k,t}\,n^{t+1}$ for all $n\ge1$. For two colours,
$F(n)\ge n^{t}$ whenever $\log_2n\ge(t+2)^{2}\,2^{t}$.
\end{theorem}

The upper bound comes from a random colouring and the Sauer--Shelah
lemma.

\begin{theorem}
\label{thm:upper}
Let $d\ge1$ be an integer with $2^{d}>nd+2$. Then
\[
F(n)\ \le\ \sum_{j=0}^{d-1}\binom{n}{j}.
\]
In particular, for $n\ge2$ and
$d=\lfloor\log_2n\rfloor+\lfloor\log_2\lfloor\log_2n\rfloor\rfloor+3$,
\[
F(n)\ \le\ d\,n^{d-1}\ \le\ \bigl(\log_2n+\log_2\log_2n+3\bigr)\,
n^{\log_2n+\log_2\log_2n+2}.
\]
\end{theorem}

Together the two theorems settle both parts of \eqref{eq:conj}, and they
locate the exponent: the largest $\omega(n)$ with $F(n)\ge n^{\omega(n)}$
lies between $(1-o(1))\log_2\log_2n$ and $(1+o(1))\log_2n$.

\begin{corollary}[The Erd\H{o}s--Ulam conjectures]
\label{cor:conj}
Both statements in \eqref{eq:conj} hold, in the following form.
\begin{enumerate}[label=\textup{(\roman*)},leftmargin=2em]
\item For every $n\ge2$ we have $F(n)\ge n^{\omega(n)}$, where
\[
\omega(n)=\ell-2\log_2(\ell+2)-1,\qquad \ell=\log_2\log_2n ,
\]
and $\omega(n)\to\infty$.
\item There is a function $\eps:\mathbb{N}\to\mathbb{R}$ with
$\eps(n)\to0$ such that $F(n)<(1+\eps(n))^{n}$ for every $n\ge1$.
\item If $\omega:\mathbb{N}\to\mathbb{R}$ satisfies $F(n)\ge
n^{\omega(n)}$ for every $n$, then for every $\eps>0$ we have
$\omega(n)\le(1+\eps)\log_2n$ for all large $n$.
\item For every $k\ge2$ there are functions $\omega_k(n)\to\infty$ and
$\eps_k(n)\to0$ such that $n^{\omega_k(n)}\le F_k(n)<(1+\eps_k(n))^{n}$
for every $n\ge1$.
\end{enumerate}
\end{corollary}

Parts (i) and (ii) are the two statements of \eqref{eq:conj}, with the
term $o(1)$ written out as a function $\eps(n)\to0$, so
Corollary~\ref{cor:conj} answers both questions of Problem~1183 in the
affirmative. Since $\log_2(\ell+2)=o(\ell)$, the exponent in (i) is
$(1-o(1))\log_2\log_2n$. This is also, to first order, the size of the
monochromatic union-closed families that we can exhibit in a uniformly
random colouring (Proposition~\ref{prop:random}), so a better lower bound
for $F(n)$ would have to say something new about random colourings as
well. In the other direction, Proposition~\ref{prop:count} reduces an
upper bound for the random colouring to counting union-closed families.

The two functions of Erd\H{o}s and Ulam turn out to be of quite
different sizes. For sublattices the trivial bound \eqref{eq:trivial} is
correct up to a factor of order $(\log n)^{2}$.

\begin{theorem}
\label{thm:f}
For every $n\ge1$, with $L=\lfloor\log_2n\rfloor+1$,
\[
\Bigl\lceil\frac{n+1}{2}\Bigr\rceil\ \le\ f(n)\ \le\
\min\bigl\{\,n^{2}+n+1,\ \ nL(3L+2)+2n+3L+3\,\bigr\}.
\]
Thus $f(n)\le(3+o(1))\,n(\log_2n)^{2}$, whereas $F(n)/n^{C}\to\infty$ for
every constant $C$.
\end{theorem}

\subsection{Colourings by cardinality, Hilbert cubes and progressions}

Howorka's case of the problem is arithmetic in nature. Call $\chi$ a
\emph{cardinality colouring} if $\chi(S)=g(|S|)$ for some
$g:\{0,1,\dots,n\}\to[k]$. The building blocks of our families are the
\emph{$t$-cubes} $\{U\cup X_I:\ I\subseteq[t]\}$, where $U,X_1,\dots,X_t$
are pairwise disjoint, the $X_i$ are nonempty and
$X_I=\bigcup_{i\in I}X_i$. The sizes of the members of such a cube form
the \emph{Hilbert cube}
\[
H(u;a_1,\dots,a_t)=\Bigl\{u+\sum_{i\in I}a_i:\ I\subseteq[t]\Bigr\},
\qquad u=|U|,\ a_i=|X_i|,
\]
so a cube is monochromatic under a cardinality colouring exactly when $g$
is constant on the corresponding Hilbert cube. Hilbert's lemma
\cite{Hi1892}, proved in the course of Hilbert's work on the
irreducibility of polynomials with integer coefficients and usually
regarded as the first theorem of Ramsey theory, says that every finite
colouring of the positive integers has monochromatic Hilbert cubes of
every dimension. Our argument gives it with an explicit window
(Corollary~\ref{cor:hilbertlemma}): every $k$-colouring of
$\{r,r+1,\dots,r+t^{2}k^{2^{t}}\}$ is constant on a $t$-dimensional
Hilbert cube inside this interval. Read through the dictionary between
cubes of sets and cubes of integers, the proof of Theorem~\ref{thm:lower}
produces many monochromatic Hilbert cubes spread along a single sequence
of bounded steps.

\begin{theorem}
\label{thm:hilbert}
Let $k\ge2$ and $t\ge1$, put $M=t^{2}k^{2^{t}}$, let $n\ge4M(t+1)$, and
put $m=\lfloor(n-1)/M\rfloor$ and $s=\lfloor m/2\rfloor$. For every
$g:\{0,1,\dots,n\}\to[k]$ there are integers
$\beta_1,\dots,\beta_m\in[M]$ with $\beta_1+\dots+\beta_m<n$ and a colour
$c$ with the following property. For at least
\[
\frac{n^{t+1}}{k\,t^{t}\,(4M^{2})^{t+1}}
\]
of the pairs $(j,A)$ with $s<j\le m$, $A\subseteq[s]$ and $|A|=t$, the
function $g$ takes the value $c$ at every point of the Hilbert cube
\[
H_\beta(j,A)=\Bigl\{\beta_1+\beta_2+\dots+\beta_j-\sum_{b\in B}\beta_b:\
B\subseteq A\Bigr\}.
\]
\end{theorem}

Different pairs $(j,A)$ may give the same set of integers, and there are
at most $(n+1)M^{t}$ Hilbert cubes $H(u;a_1,\dots,a_t)$ in
$\{0,\dots,n\}$ with all $a_i\le M$. What Theorem~\ref{thm:hilbert}
provides is one sequence $\beta$ for which a fixed proportion of the cubes
indexed by the pairs $(j,A)$ have a common colour, all of them hanging
from the partial sums $\beta_1+\dots+\beta_j$. This is the structure that
lets the corresponding families of sets be glued into one union-closed
family. Section~\ref{sec:card} also shows that a monochromatic arithmetic
progression of sizes yields a large monochromatic union-closed family and
a monochromatic sublattice, and derives Howorka's theorem from van der
Waerden's theorem with an explicit bound (Theorem~\ref{thm:howorka}).

\subsection{The idea of the proof of Theorem~\ref{thm:lower}}

Chains and cubes are monochromatic union-closed families that are easy to
find, but a chain has at most $n+1$ members and a single cube is too
rigid to be forced to be large. The families we build are unions of many
cubes that fit together. Take pairwise disjoint nonempty blocks
$Y_1,\dots,Y_m$, put $C_j=Y_1\cup\dots\cup Y_j$, and for $j>s$ and
$A\subseteq[s]$ consider the cube $Q(j,A)=\{C_j\setminus Y_B:\ B\subseteq
A\}$. The union of a member of $Q(j,A)$ and a member of $Q(j',A')$ with
$j\le j'$ is again a member of $Q(j',A')$, so the union of any collection
of cubes $Q(j,A)$ that are monochromatic in a common colour is
union-closed. There are about $m^{t+1}$ pairs $(j,A)$ with $|A|=t$, and it
is therefore enough to make a fixed proportion of these cubes
monochromatic.

For a fixed chain of blocks this cannot be arranged, since the colouring
may have been chosen against it. We therefore randomise twice: the ground
set is permuted by a uniformly random $\sigma$, and each block size is
chosen at random from $\{1,\dots,M\}$. After a random permutation, the
probability that a cube is monochromatic depends only on the sizes of its
parts (Lemma~\ref{lem:shape}). The central step is a supersaturation
statement: for every $r$, some shape whose base size lies between $r$ and
$r+M$ is monochromatic with probability at least $k^{-2^{t}}$
(Proposition~\ref{prop:saturation}). It follows from an inequality of
Cauchy--Schwarz type for products over the vertices of a box
(Lemma~\ref{lem:box}), applied to the sets obtained by running along $t$
disjoint chains of elements (Figure~\ref{fig:box}). The shape of $Q(j,A)$
is governed by $t+1$ block sizes, those of $Y_j$ and of $Y_a$ for $a\in
A$, and these can be tuned independently of all the others
(Lemma~\ref{lem:slices}). Averaging over $\sigma$ and the block sizes then
produces one chain of blocks in which a proportion $k^{-2^{t}}M^{-t-1}$ of
all the cubes is monochromatic (Lemma~\ref{lem:average}).
Figure~\ref{fig:proofmap} shows how the pieces are assembled.

\begin{figure}[t]
\centering
\begin{tikzpicture}[
    box/.style={draw=black!70, line width=0.6pt, rounded corners=3pt, fill=#1,
      text width=10cm, align=flush left, inner sep=5pt, font=\small},
    box/.default=white,
    arr/.style={-{Stealth[length=6pt]}, line width=0.8pt}]
  \node[box=gray!8] (box) {\textbf{Lemma~\ref{lem:box}.} For $f\ge0$ on $G^{t}$, the product of
    $f$ over the $2^{t}$ vertices of a random box has mean at least $(\mathbb{E}f)^{2^{t}}$; so a
    random box in $\{0,\dots,L\}^{t}$ is monochromatic with probability at least $k^{1-2^{t}}$.};
  \node[box=teal!8, below=8mm of box] (sat) {\textbf{Proposition~\ref{prop:saturation}.} For every
    $r\le n-M$ some shape $(u;a_1,\dots,a_t)$ with $r\le u\le r+M-t$ satisfies
    $\rho_\chi(u;a)\ge k^{-2^{t}}$.};
  \node[box=teal!13, below=8mm of sat] (slice) {\textbf{Lemma~\ref{lem:slices}.} For each level $j$
    and generator set $A$, at least $M^{m-t-1}$ of the $M^m$ size vectors $\beta$ give the cube
    $Q(j,A)$ a good shape.};
  \node[box=teal!18, below=8mm of slice] (avg) {\textbf{Lemma~\ref{lem:average}.} Some ordering
    $\sigma$ and sizes $\beta$ make at least \mbox{$|\mathcal P|/(k^{2^{t}}M^{t+1})$} of the cubes
    monochromatic at once.};
  \node[box=teal!24, below=8mm of avg] (glue) {\textbf{Proposition~\ref{prop:glue}.} Monochromatic
    cubes of one block chain glue into a monochromatic union-closed family, so
    \mbox{$F_k(n)\ge|\mathcal P|/(k^{2^{t}+1}M^{t+1})\ge c_{k,t}\,n^{t+1}$}.};
  \foreach \a/\b/\t in {box/sat/{run along $t$ chains of $L$ points; discard degenerate boxes},
      sat/slice/{tune $t+1$ block sizes, fix the rest},
      slice/avg/{average over $(\sigma,\beta)$},
      avg/glue/{union rule, Lemma~\ref{lem:union}}}
    \draw[arr] ($(\a.south west)!0.14!(\a.south east)$) --
      node[right=4pt, font=\footnotesize\itshape, text=black!70] {\t}
      ($(\b.north west)!0.14!(\b.north east)$);
\end{tikzpicture}
\caption{The structure of the proof of Theorem~\ref{thm:lower}. Here
$\cP$ is the set of pairs $(j,A)$ with $s<j\le m$, $A\subseteq[s]$ and
$|A|=t$, and $\rho_\chi$ is defined after Lemma~\ref{lem:shape}.}
\label{fig:proofmap}
\end{figure}
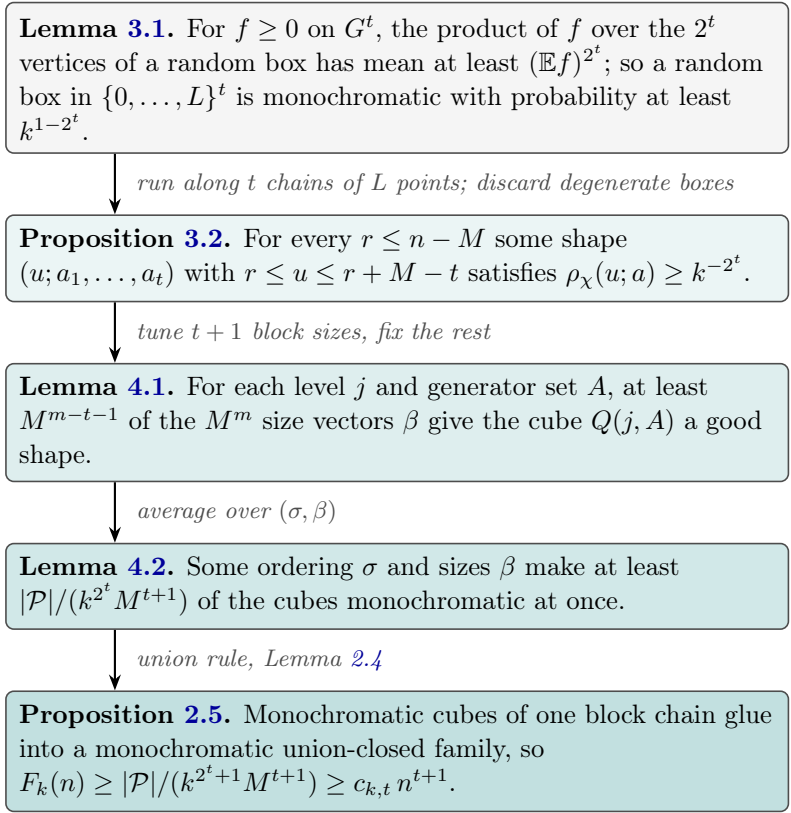

\subsection{Related work}

The union-closed families in this paper should not be confused with
those of Frankl's union-closed sets conjecture, which asks whether every
union-closed family other than $\{\emptyset\}$ has an element lying in at
least half of its members. That problem concerns the internal structure
of one family and involves no colouring. Gilmer \cite{Gi22} proved that
some element lies in a constant proportion of the members, and Alweiss,
Huang and Sellke \cite{AHS24} improved the constant. The entropy method
behind these results compares the union of two independent random
members of a family with a single random member, and we see no analogue
of it here, where the family is not given in advance.

Families of $2^{d}$ sets of the form $\{U\cup X_I\}$ are called
$d$-dimensional Boolean algebras by Gunderson, R\"odl and Sidorenko
\cite{GRS99}, who studied their extremal theory alongside that of
$d$-dimensional affine cubes of integers and complete $d$-partite
$d$-uniform hypergraphs. The passage between cubes of sets and Hilbert
cubes of integers used in Section~\ref{sec:card} is of the same kind.
Proposition~\ref{prop:saturation} contains the two-letter case of the
multidimensional Hales--Jewett theorem \cite{HJ63} with an explicit bound
(Remark~\ref{rem:hj}); see \cite{GRS90} for this theorem and for the
background in Ramsey theory, including van der Waerden's theorem
\cite{vdW27}. The box inequality behind it is the simplest instance of the
repeated use of the Cauchy--Schwarz inequality introduced by Gowers
\cite{Go01}. The upper bound rests on the Sauer--Shelah lemma
\cite{Sa72,Sh72}, which is also at the root of the theory of
Vapnik--Chervonenkis dimension \cite{VC71}.

\subsection{Formal verification}

All theorems, propositions, corollaries and lemmas of
Sections~\ref{sec:cubes}--\ref{sec:f} have been formalised and checked in
the Lean proof assistant \cite{MU21} with the Mathlib library
\cite{Mathlib20}. The functions $f$, $F$ and $F_k$ are defined in the
formal development literally as above, the two conjectures are stated in
the form \eqref{eq:conj} with the term $o(1)$ written out as a function
tending to zero, and Lean reports that the proof of each formal statement
uses only the three standard axioms of the system.
Section~\ref{sec:lean} describes the formalisation.

\subsection{Organisation}

Section~\ref{sec:cubes} introduces cubes and block chains and proves that
monochromatic cubes of one chain can be glued. Section~\ref{sec:cube}
proves the box inequality and the supersaturation statement, and
Section~\ref{sec:lower} completes the proof of Theorem~\ref{thm:lower}.
Section~\ref{sec:upper} proves Theorem~\ref{thm:upper} and
Corollary~\ref{cor:conj} and studies the random colouring.
Section~\ref{sec:card} treats cardinality colourings, Hilbert cubes and
arithmetic progressions, and Section~\ref{sec:f} the lattice function $f$.
Section~\ref{sec:lean} describes the formalisation, and
Section~\ref{sec:questions} collects open questions.

\section{Cubes and block chains}
\label{sec:cubes}

Throughout, $k\ge2$ and $\chi$ is a $k$-colouring of $2^{[n]}$, and $S_n$
is the group of permutations of $[n]$. A permutation acts on sets and on
families of sets in the obvious way: $\sigma(S)=\{\sigma(x):x\in S\}$ and
$\sigma(\cF)=\{\sigma(S):S\in\cF\}$. For sets $X_1,X_2,\dots$ and a finite
set $I$ of indices we write $X_I=\bigcup_{i\in I}X_i$, so that
$X_\emptyset=\emptyset$.

\begin{definition}
\label{def:cube}
Let $t\ge0$. A \emph{$t$-cube} in $2^{[n]}$ is a family
\[
Q(U;X_1,\dots,X_t)=\{U\cup X_I:\ I\subseteq[t]\},
\]
where $U,X_1,\dots,X_t\subseteq[n]$ are pairwise disjoint and
$X_1,\dots,X_t$ are nonempty. We call $U$ the \emph{base},
$T=U\cup X_{[t]}$ the \emph{top} and $X_1,\dots,X_t$ the
\emph{generators}. The \emph{shape} of the cube is the tuple
$(|U|;|X_1|,\dots,|X_t|)$.
\end{definition}

The $2^{t}$ sets $U\cup X_I$ are distinct, and
\[
(U\cup X_I)\cup(U\cup X_J)=U\cup X_{I\cup J},\qquad
(U\cup X_I)\cap(U\cup X_J)=U\cup X_{I\cap J},
\]
so a $t$-cube is a sublattice of $2^{[n]}$ isomorphic to $2^{[t]}$.
Writing $U\cup X_I=T\setminus X_{[t]\setminus I}$ describes the same
cube from the top downwards, as the family $\{T\setminus X_I:\
I\subseteq[t]\}$; this is the form in which cubes appear in block chains.
Figure~\ref{fig:holecube} shows a $3$-cube in this form.

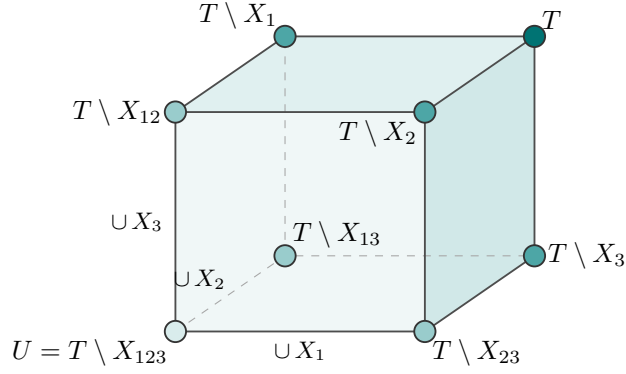
\begin{figure}[t]
\centering
\begin{tikzpicture}[
    x={(3.3cm,0cm)}, y={(1.45cm,1.0cm)}, z={(0cm,2.9cm)},
    v/.style={circle, draw=black!80, line width=0.7pt, inner sep=0pt, minimum size=8pt},
    lab/.style={font=\small, inner sep=2pt},
    elab/.style={font=\footnotesize, inner sep=1.5pt, fill=white, fill opacity=0.85, text opacity=1}]
  \fill[teal!6]  (0,0,0) -- (1,0,0) -- (1,0,1) -- (0,0,1) -- cycle;
  \fill[teal!12] (0,0,1) -- (1,0,1) -- (1,1,1) -- (0,1,1) -- cycle;
  \fill[teal!18] (1,0,0) -- (1,1,0) -- (1,1,1) -- (1,0,1) -- cycle;
  \draw[dashed, gray!70] (0,0,0) -- (0,1,0) -- (1,1,0);
  \draw[dashed, gray!70] (0,1,0) -- (0,1,1);
  \draw[black!70, line width=0.7pt]
    (0,0,0) -- (1,0,0) -- (1,1,0) -- (1,1,1) -- (0,1,1) -- (0,0,1) -- cycle
    (1,0,0) -- (1,0,1) -- (1,1,1)
    (0,0,1) -- (1,0,1);
  \node[v, fill=teal!15] (U)   at (0,0,0) {};
  \node[v, fill=teal!40] (x1)  at (1,0,0) {};
  \node[v, fill=teal!40] (x2)  at (0,1,0) {};
  \node[v, fill=teal!40] (x3)  at (0,0,1) {};
  \node[v, fill=teal!70] (x12) at (1,1,0) {};
  \node[v, fill=teal!70] (x13) at (1,0,1) {};
  \node[v, fill=teal!70] (x23) at (0,1,1) {};
  \node[v, fill=teal!90!black] (T) at (1,1,1) {};
  \node[lab, below left=1pt]  at (U)   {$U=T\setminus X_{123}$};
  \node[lab, below right=1pt] at (x1)  {$T\setminus X_{23}$};
  \node[lab, above right=1pt] at (x2)  {$T\setminus X_{13}$};
  \node[lab, left=3pt]        at (x3)  {$T\setminus X_{12}$};
  \node[lab, right=3pt]       at (x12) {$T\setminus X_{3}$};
  \node[lab, below left=1pt]  at (x13) {$T\setminus X_{2}$};
  \node[lab, above left=1pt]  at (x23) {$T\setminus X_{1}$};
  \node[lab, above right=1pt] at (T)   {$T$};
  \node[elab, below=2pt] at (0.5,0,0) {$\cup\,X_1$};
  \node[elab, left=3pt]  at (0,0,0.5) {$\cup\,X_3$};
  \node[elab, fill=none, above left=0pt] at (0,0.5,0) {$\cup\,X_2$};
\end{tikzpicture}
\caption{A $3$-cube with top $T$ and generators $X_1,X_2,X_3$, where
$X_{12}=X_1\cup X_2$ and so on. Moving along an edge in direction $i$
adds or removes $X_i$. The vertex $U=T\setminus X_{123}$ is the base.}
\label{fig:holecube}
\end{figure}

\begin{lemma}
\label{lem:shape}
If $Q$ and $Q'$ are $t$-cubes in $2^{[n]}$ of the same shape, then
\begin{multline*}
\bigl|\{\sigma\in S_n:\ \sigma(Q)\text{ is monochromatic}\}\bigr|\\
=\bigl|\{\sigma\in S_n:\ \sigma(Q')\text{ is monochromatic}\}\bigr| .
\end{multline*}
\end{lemma}

\begin{proof}
Let $Q=Q(U;X_1,\dots,X_t)$ and $Q'=Q(U';X'_1,\dots,X'_t)$ with tops $T$
and $T'$. The sets $U,X_1,\dots,X_t,[n]\setminus T$ partition $[n]$, as do
$U',X'_1,\dots,X'_t,[n]\setminus T'$, and corresponding parts have the
same size. So there is $\tau\in S_n$ with $\tau(U)=U'$ and
$\tau(X_i)=X'_i$ for every $i$, and then $\tau(Q)=Q'$. Now
$\sigma(Q')=(\sigma\tau)(Q)$, and $\sigma\mapsto\sigma\tau$ is a bijection
of $S_n$.
\end{proof}

For integers $u\ge0$ and $a_1,\dots,a_t\ge1$ with
$u+a_1+\dots+a_t\le n$ we may therefore define
\[
\rho_\chi(u;a_1,\dots,a_t)=\frac{1}{n!}\,
\bigl|\{\sigma\in S_n:\ \sigma(Q)\text{ is monochromatic}\}\bigr|,
\]
where $Q$ is any $t$-cube of shape $(u;a_1,\dots,a_t)$. Since $\sigma(Q)$
is a uniformly random cube of that shape when $\sigma$ is uniform,
$\rho_\chi(u;a)$ is the probability that a random cube of shape $(u;a)$ is
monochromatic.

\begin{definition}
\label{def:chain}
Let $Y_1,\dots,Y_m\subseteq[n]$ be pairwise disjoint nonempty sets, the
\emph{blocks}, and put $C_j=Y_1\cup\dots\cup Y_j$ for $0\le j\le m$. For
$1\le j\le m$ and $A\subseteq[j-1]$ let
\[
Q(j,A)=\{C_j\setminus Y_B:\ B\subseteq A\}.
\]
This is the $|A|$-cube with top $C_j$, generators $Y_a$ $(a\in A)$ and base
$C_j\setminus Y_A$; its shape is
$\bigl(\sum_{i\le j,\,i\notin A}|Y_i|;\ (|Y_a|)_{a\in A}\bigr)$, the
generators being listed in increasing order of $a$.
\end{definition}

Figure~\ref{fig:chain} shows a member of such a cube. The chain
structure enters through the following rule, illustrated in
Figure~\ref{fig:union}.

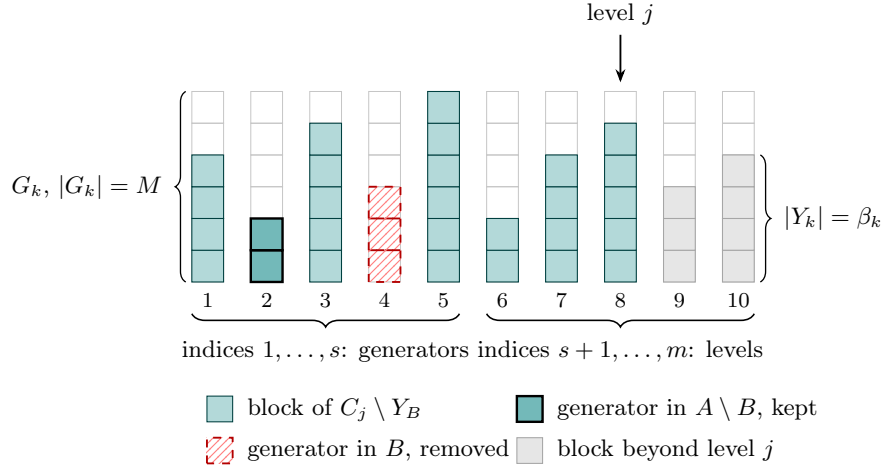
\begin{figure}[t]
\centering
\begin{tikzpicture}[cell/.style={minimum size=0.42cm, inner sep=0pt, outer sep=0pt}]
  \def\M{6}
  \def\colsep{0.78}
  \foreach \k/\b in {1/4,2/2,3/5,4/3,5/6,6/2,7/4,8/5,9/3,10/4} {
    \pgfmathsetmacro{\xx}{(\k-1)*\colsep}
    \foreach \r in {1,...,\M} {
      \draw[gray!45, line width=0.3pt] (\xx,{(\r-1)*0.42}) rectangle ++(0.42,0.42);
    }
    \ifnum\k=4
      \foreach \r in {1,...,\b} {
        \draw[red!70!black, dashed, line width=0.7pt, pattern=north east lines,
          pattern color=red!50] (\xx,{(\r-1)*0.42}) rectangle ++(0.42,0.42);
      }
    \else\ifnum\k=2
      \foreach \r in {1,...,\b} {
        \draw[black, line width=0.9pt, fill=teal!55] (\xx,{(\r-1)*0.42}) rectangle ++(0.42,0.42);
      }
    \else\ifnum\k>8
      \foreach \r in {1,...,\b} {
        \draw[gray!70, line width=0.4pt, fill=gray!20] (\xx,{(\r-1)*0.42}) rectangle ++(0.42,0.42);
      }
    \else
      \foreach \r in {1,...,\b} {
        \draw[teal!50!black, line width=0.4pt, fill=teal!30] (\xx,{(\r-1)*0.42}) rectangle ++(0.42,0.42);
      }
    \fi\fi\fi
    \node[font=\scriptsize, below] at ({\xx+0.21},-0.02) {$\k$};
  }
  \draw[decorate, decoration={brace, amplitude=5pt, mirror}, line width=0.6pt]
    (0,-0.42) -- ({4*\colsep+0.42},-0.42)
    node[midway, below=6pt, font=\footnotesize] {indices $1,\dots,s$: generators};
  \draw[decorate, decoration={brace, amplitude=5pt, mirror}, line width=0.6pt]
    ({5*\colsep},-0.42) -- ({9*\colsep+0.42},-0.42)
    node[midway, below=6pt, font=\footnotesize] {indices $s+1,\dots,m$: levels};
  \draw[-{Stealth[length=5pt]}, line width=0.7pt] ({7*\colsep+0.21},{\M*0.42+0.75}) --
    ({7*\colsep+0.21},{\M*0.42+0.12}) node[pos=0, above, font=\footnotesize] {level $j$};
  \draw[decorate, decoration={brace, amplitude=4pt}, line width=0.5pt]
    ({-0.08},0) -- ({-0.08},{\M*0.42}) node[midway, left=5pt, font=\footnotesize] {$G_k$, $|G_k|=M$};
  \draw[decorate, decoration={brace, amplitude=4pt, mirror}, line width=0.5pt]
    ({9*\colsep+0.5},0) -- ({9*\colsep+0.5},{4*0.42}) node[midway, right=5pt, font=\footnotesize] {$|Y_k|=\beta_k$};
  \begin{scope}[shift={(0.2,-1.85)}]
    \draw[teal!50!black, fill=teal!30, line width=0.4pt] (0,0) rectangle ++(0.34,0.34);
    \node[right, font=\footnotesize] at (0.4,0.17) {block of $C_j\setminus Y_B$};
    \draw[black, fill=teal!55, line width=0.9pt] (4.1,0) rectangle ++(0.34,0.34);
    \node[right, font=\footnotesize] at (4.5,0.17) {generator in $A\setminus B$, kept};
    \draw[red!70!black, dashed, pattern=north east lines, pattern color=red!50, line width=0.7pt]
      (0,-0.55) rectangle ++(0.34,0.34);
    \node[right, font=\footnotesize] at (0.4,-0.38) {generator in $B$, removed};
    \draw[gray!70, fill=gray!20, line width=0.4pt] (4.1,-0.55) rectangle ++(0.34,0.34);
    \node[right, font=\footnotesize] at (4.5,-0.38) {block beyond level $j$};
  \end{scope}
\end{tikzpicture}
\caption{A block chain with $m=10$ blocks placed in slots $G_1,\dots,G_{10}$
of size $M=6$, with $|Y_k|=\beta_k$ and
$\beta=(4,2,5,3,6,2,4,5,3,4)$. Here $s=5$, and the teal blocks form the
member $C_8\setminus Y_{\{4\}}$ of the cube $Q(8,\{2,4\})$; the hatched
block $Y_4$ is removed and the grey blocks lie beyond $C_8$.}
\label{fig:chain}
\end{figure}

\begin{lemma}[Union rule]
\label{lem:union}
Let $1\le j\le j'\le m$ and $B,B'\subseteq[j]$. Then
\[
(C_j\setminus Y_B)\cup(C_{j'}\setminus Y_{B'})=C_{j'}\setminus
Y_{B\cap B'} .
\]
\end{lemma}

\begin{proof}
Both sides are contained in $C_{j'}$. Let $x\in C_{j'}$. If $x\notin
Y_{B'}$, then $x$ lies in $C_{j'}\setminus Y_{B'}$ and also outside
$Y_{B\cap B'}$. Otherwise $x\in Y_b$ for some $b\in B'$; then $b\le j$, so
$x\in C_j$, and $x$ lies in the left side if and only if $b\notin B$, that
is, if and only if $x\notin Y_{B\cap B'}$.
\end{proof}

\begin{figure}[t]
\centering
\begin{tikzpicture}[x=0.62cm, y=0.62cm]
  \newcommand{\blockrow}[3]{
    \foreach \k in {1,...,9} {
      \pgfmathtruncatemacro{\v}{#2[\k-1]}
      \ifnum\v=1
        \draw[teal!50!black, fill=teal!35, line width=0.4pt] (\k-1,#1) rectangle ++(0.8,0.8);
      \else
        \draw[gray!50, line width=0.4pt] (\k-1,#1) rectangle ++(0.8,0.8);
      \fi
    }
    \node[left, font=\small] at (-0.3,#1+0.4) {#3};
  }
  \blockrow{4}{{1,0,1,0,1,1,0,0,0}}{$C_6\setminus Y_{\{2,4\}}$}
  \blockrow{2.6}{{1,1,0,0,1,1,1,1,0}}{$\cup\;\;C_8\setminus Y_{\{3,4\}}$}
  \draw[line width=0.5pt] (-0.1,2.25) -- (8.9,2.25);
  \blockrow{1.2}{{1,1,1,0,1,1,1,1,0}}{$=\;\;C_8\setminus Y_{\{4\}}$}
  \foreach \k in {1,...,9} { \node[font=\scriptsize] at (\k-0.6,0.75) {$\k$}; }
  \draw[dashed, line width=0.5pt] (4.9,0.5) -- (4.9,5.1);
  \node[font=\footnotesize, above] at (4.9,5.1) {$s=5$};
  \node[font=\footnotesize, align=left, right] at (9.3,1.6)
    {$\{2,4\}\cap\{3,4\}=\{4\}$};
\end{tikzpicture}
\caption{The union rule with $s=5$: $C_6\setminus Y_{\{2,4\}}$ and
$C_8\setminus Y_{\{3,4\}}$ have union $C_8\setminus Y_{\{4\}}$. Filled
squares are blocks present in the set.}
\label{fig:union}
\end{figure}
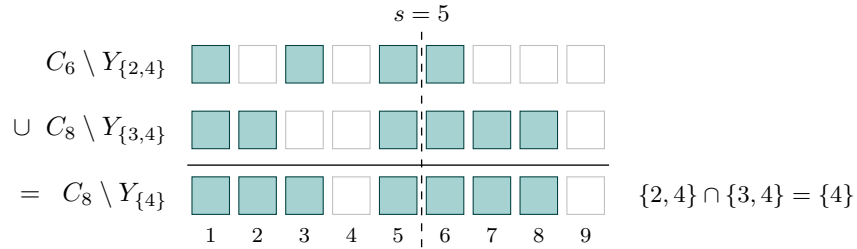

\begin{proposition}[Gluing]
\label{prop:glue}
Let $1\le s<m$, and let $\cS$ be a set of pairs $(j,A)$ with $s<j\le m$
and $A\subseteq[s]$ such that the cube $Q(j,A)$ is monochromatic for every
$(j,A)\in\cS$. Then there is a monochromatic union-closed family with at
least $|\cS|/k$ members.
\end{proposition}

\begin{proof}
Choose a colour $c$ such that the set $\cS_c$ of pairs in $\cS$ whose cube
has colour $c$ satisfies $|\cS_c|\ge|\cS|/k$, and let $\cF$ be the union
of the cubes $Q(j,A)$ over $(j,A)\in\cS_c$. Every member of $\cF$ has
colour $c$.

Let $C_j\setminus Y_B\in Q(j,A)$ and $C_{j'}\setminus Y_{B'}\in Q(j',A')$
be members of $\cF$, where $(j,A),(j',A')\in\cS_c$ and, say, $j\le j'$.
As $B,B'\subseteq[s]\subseteq[j]$, Lemma~\ref{lem:union} shows that
their union is $C_{j'}\setminus Y_{B\cap B'}$, and since $B\cap
B'\subseteq A'$ this set belongs to $Q(j',A')\subseteq\cF$. So $\cF$ is
union-closed.

Finally, the map $(j,A)\mapsto C_j\setminus Y_A$ from $\cS_c$ to $\cF$ is
injective. Indeed, $j>s$ means $j\notin A$, so $C_j\setminus Y_A$ meets
$Y_j$ and no block of larger index, which determines $j$; and then $A$ is
the set of indices $i\le s$ for which $Y_i$ is disjoint from
$C_j\setminus Y_A$. Hence $|\cF|\ge|\cS_c|\ge|\cS|/k$.
\end{proof}

The family $\cF$ is not closed under intersection in general: the
intersection of the two sets above is $C_j\setminus Y_{B\cup B'}$, and
$B\cup B'$ need not be contained in $A$. This is as it must be, in view of
Theorem~\ref{thm:f}.

\section{A box inequality and supersaturation}
\label{sec:cube}

The proof of Theorem~\ref{thm:lower} needs, for every $r$, a shape whose
base size lies a little above $r$ and which is monochromatic with
probability bounded below by a quantity depending on $k$ and $t$ alone.
This section obtains such a shape from an inequality of Cauchy--Schwarz
type for products taken over the vertices of a box.

Let $G$ be a finite set and $t\ge0$. For $x,x'\in G^{t}$ and
$\omega\in\{0,1\}^{t}$ let $x^{(\omega)}\in G^{t}$ be the point whose
$i$-th coordinate is $x'_i$ if $\omega_i=1$ and $x_i$ if $\omega_i=0$. The
$2^{t}$ points $x^{(\omega)}$ are the vertices of the \emph{box} spanned by
$x$ and $x'$; they are distinct precisely when $x_i\ne x'_i$ for every
$i$. Expectations are taken with respect to the uniform distribution on
the finite set indicated by the variables.

\begin{lemma}[Box inequality]
\label{lem:box}
Let $f:G^{t}\to[0,\infty)$. If $x$ and $x'$ are independent and uniformly
distributed in $G^{t}$, then
\[
\Ex\prod_{\omega\in\{0,1\}^{t}}f\bigl(x^{(\omega)}\bigr)\ \ge\
\bigl(\Ex f(x)\bigr)^{2^{t}} .
\]
\end{lemma}

\begin{proof}
We use induction on $t$. For $t=0$ both sides equal the value of $f$ at
the only point of $G^{0}$. Let $t\ge1$ and write the points of $G^{t}$ as
$(a,y)$ with $a\in G$ and $y\in G^{t-1}$. For $y,y'\in G^{t-1}$ and $a\in
G$ put
\[
h_{y,y'}(a)=\prod_{\omega\in\{0,1\}^{t-1}}f\bigl(a,y^{(\omega)}\bigr).
\]
Splitting $\omega\in\{0,1\}^{t}$ according to its first coordinate shows
that the product over the box spanned by $x=(a,y)$ and $x'=(a',y')$ equals
$h_{y,y'}(a)\,h_{y,y'}(a')$. The variables $a,a',y,y'$ are independent,
so the Cauchy--Schwarz inequality gives
\[
\Ex\prod_{\omega}f\bigl(x^{(\omega)}\bigr)
=\Ex_{y,y'}\bigl(\Ex_a h_{y,y'}(a)\bigr)^{2}
\ \ge\ \bigl(\Ex_{a}\Ex_{y,y'}h_{y,y'}(a)\bigr)^{2}.
\]
For fixed $a$, the induction hypothesis applied to $f(a,\cdot)$ gives
$\Ex_{y,y'}h_{y,y'}(a)\ge(\Ex_yf(a,y))^{2^{t-1}}$, and since $z\mapsto
z^{2^{t-1}}$ is convex on $[0,\infty)$, Jensen's inequality gives
\[
\Ex_a\bigl(\Ex_yf(a,y)\bigr)^{2^{t-1}}\ \ge\
\bigl(\Ex_a\Ex_yf(a,y)\bigr)^{2^{t-1}}=\bigl(\Ex f\bigr)^{2^{t-1}}.
\]
Squaring, the left-hand side of the lemma is at least
$(\Ex f)^{2^{t}}$.
\end{proof}

The quantity on the left of Lemma~\ref{lem:box} is the $2^{t}$-th power
of the box norm of $f$, and the argument is the one Gowers used for the
uniformity norms in a proof of Szemer\'edi's theorem
\cite{Go01}.\footnote{In that setting one works with a function on an
abelian group and the vertices of a parallelepiped $x+\omega\cdot h$, and
repeated use of the Cauchy--Schwarz inequality compares the average over
parallelepipeds with the average of the function. Here the underlying
set is a product and the box is a combinatorial one, which is the
simplest case of the method.} Applied to the indicator of each colour
class it says that a random box is monochromatic with probability at
least $k^{1-2^{t}}$, whatever the colouring. What remains is to turn boxes
in a grid into cubes of sets, and this is where the order of the ground
set is randomised.

For the rest of the paper put
\begin{equation}
\label{eq:Mt}
L_k(t)=t\,k^{2^{t}}\qquad\text{and}\qquad M_k(t)=t\,L_k(t)=t^{2}\,k^{2^{t}}
\qquad(t\ge1),
\end{equation}
so that $M_2(1)=4$, $M_2(2)=64$ and $M_2(3)=2304$. Note that
$M_k(t)\ge t+1$.

\begin{proposition}[Supersaturation]
\label{prop:saturation}
Let $t\ge1$ and $L\ge1$ be integers, put $M=tL$, and let $r\ge0$ be an
integer with $r+M\le n$. There are integers $u$ and $a_1,\dots,a_t\ge1$
with
\[
r\le u,\qquad (u-r)+a_1+\dots+a_t\le M,\qquad
\rho_\chi(u;a_1,\dots,a_t)\ \ge\ k^{1-2^{t}}-\frac{t}{L+1}.
\]
In particular, with $L=L_k(t)$ and $M=M_k(t)$ as in \eqref{eq:Mt}, one
can achieve $\rho_\chi(u;a_1,\dots,a_t)\ge k^{-2^{t}}$, and then $u\le r+M-t$ and
$a_i\le M$ for every $i$.
\end{proposition}

\begin{proof}
Choose pairwise disjoint sets $R,Z_1,\dots,Z_t\subseteq[n]$ with $|R|=r$
and $|Z_i|=L$, which is possible because $r+tL\le n$, and list
$Z_i=\{z_{i,1},\dots,z_{i,L}\}$. Put $G=\{0,1,\dots,L\}$ and, for $x\in
G^{t}$,
\[
S(x)=R\cup\bigcup_{i=1}^{t}\{z_{i,q}:\ 1\le q\le x_i\},
\]
so that $x$ records how far along each of the $t$ chains
$z_{i,1},z_{i,2},\dots$ the set $S(x)$ reaches. Let $\sigma\in S_n$ and
$x,x'\in G^{t}$ be independent and uniform, and let $E$ be the event
that $\chi(\sigma(S(\cdot)))$ is constant on the box spanned by $x$ and
$x'$.

Fix $\sigma$, and for each colour $c$ let $p_c$ be the probability that
$\chi(\sigma(S(x)))=c$ for uniform $x$. Lemma~\ref{lem:box} applied to the
indicator function of the colour class of $c$ shows that, conditionally
on $\sigma$, all vertices of the box have colour $c$ with probability at
least $p_c^{2^{t}}$. Since $p_1+\dots+p_k=1$, convexity gives
$\sum_cp_c^{2^{t}}\ge k\cdot k^{-2^{t}}$, and averaging over $\sigma$ gives
$\Pr(E)\ge k^{1-2^{t}}$. The event $D$ that $x_i=x'_i$ for some $i$ has
probability at most $t/(L+1)$, so
\[
\Pr(E\setminus D)\ \ge\ k^{1-2^{t}}-\frac{t}{L+1}.
\]

Suppose now that $x_i\ne x'_i$ for every $i$, and put
$\ell_i=\min(x_i,x'_i)$ and $h_i=\max(x_i,x'_i)$. The vertices of the box
are the points whose $i$-th coordinate is $\ell_i$ or $h_i$ for every
$i$, so the sets $S(x^{(\omega)})$ form the $t$-cube
$Q_{x,x'}=Q(U;X_1,\dots,X_t)$ with
\[
U=S(\ell),\qquad X_i=\{z_{i,q}:\ \ell_i<q\le h_i\}
\]
(Figure~\ref{fig:box}). Its shape is
$(r+\sum_i\ell_i;\ h_1-\ell_1,\dots,h_t-\ell_t)$, and $u=r+\sum_i\ell_i$
and $a_i=h_i-\ell_i$ satisfy $u\ge r$, $a_i\ge1$ and
$(u-r)+\sum_ia_i=\sum_ih_i\le tL=M$. The event $E\setminus D$ says that
$\sigma(Q_{x,x'})$ is monochromatic. For a fixed pair $(x,x')$ outside
$D$, the probability of this over $\sigma$ is $\rho_\chi$ at the shape of
$Q_{x,x'}$, because $\sigma(Q_{x,x'})$ is then a uniformly random cube of
that shape. Hence $\Pr(E\setminus D)$ is an average, over all pairs
$(x,x')$, of numbers that vanish on $D$ and are values of $\rho_\chi$
at admissible shapes elsewhere, and one of these values is at least
$\Pr(E\setminus D)$.

For $L=L_k(t)$ we have $t/(L+1)<t/L=k^{-2^{t}}$, so the value found is at
least $(k-1)k^{-2^{t}}\ge k^{-2^{t}}$. The bounds on $u$ and $a_i$ follow
from $a_i\ge1$.
\end{proof}

\begin{figure}[t]
\centering
\begin{tikzpicture}[x={(1.2cm,0cm)}, y={(0.66cm,0.44cm)}, z={(0cm,1.1cm)},
    dot/.style={circle, fill=black!22, inner sep=0pt, minimum size=2.3pt},
    vtx/.style={circle, draw=teal!40!black, fill=teal!45, inner sep=0pt,
      minimum size=7pt},
    tip/.style={circle, draw=red!45!black, fill=red!70!black, inner sep=0pt,
      minimum size=8pt},
    lab/.style={font=\small, fill=white, fill opacity=0.85, text opacity=1,
      inner sep=1.2pt, rounded corners=1pt}]
  \foreach \a in {0,...,4}
    \foreach \b in {0,...,4}
      \foreach \c in {0,...,4}
        \node[dot] at (\a,\b,\c) {};
  \draw[black!18, line width=0.4pt] (4,0,0) -- (4,4,0) -- (0,4,0);
  \draw[black!18, line width=0.4pt] (0,0,4) -- (4,0,4) -- (4,4,4) -- (0,4,4) -- cycle;
  \draw[black!18, line width=0.4pt] (4,0,0) -- (4,0,4) (4,4,0) -- (4,4,4) (0,4,0) -- (0,4,4);
  \draw[-{Stealth[length=5pt]}, black!65, line width=0.6pt] (0,0,0) -- (4,0,0) -- (4.7,0,0);
  \node[below right=0pt and 1pt, font=\small] at (4.7,0,0) {$x_1$};
  \draw[-{Stealth[length=5pt]}, black!65, line width=0.6pt] (4,0,0) -- (4,4.8,0);
  \node[right=2pt, font=\small] at (4,4.8,0) {$x_2$};
  \draw[-{Stealth[length=5pt]}, black!65, line width=0.6pt] (0,0,0) -- (0,0,4.75);
  \node[above, font=\small] at (0,0,4.75) {$x_3$};
  \foreach \a in {0,...,4}
    \node[below=2pt, font=\scriptsize, text=black!75] at (\a,0,0) {$\a$};
  \foreach \c in {1,...,4}
    \node[left=3pt, font=\scriptsize, text=black!75] at (0,0,\c) {$\c$};
  \foreach \b in {1,...,4}
    \node[right=3pt, font=\scriptsize, text=black!75] at (4,\b,0) {$\b$};
  \draw[teal!50!black, dashed, line width=0.6pt] (1,3,0) -- (3,3,0)
    (1,3,0) -- (1,1,0) (1,3,0) -- (1,3,2);
  \fill[teal!30, opacity=0.35] (1,1,0) -- (3,1,0) -- (3,1,2) -- (1,1,2) -- cycle;
  \fill[teal!45, opacity=0.35] (1,1,2) -- (3,1,2) -- (3,3,2) -- (1,3,2) -- cycle;
  \fill[teal!60, opacity=0.35] (3,1,0) -- (3,3,0) -- (3,3,2) -- (3,1,2) -- cycle;
  \draw[teal!40!black, line width=1.1pt]
    (1,1,0) -- (3,1,0) -- (3,3,0) -- (3,3,2) -- (1,3,2) -- (1,1,2) -- cycle
    (1,1,2) -- (3,1,2) -- (3,3,2) (3,1,0) -- (3,1,2);
  \foreach \a in {1,3}
    \foreach \b in {1,3}
      \foreach \c in {0,2}
        \node[vtx] at (\a,\b,\c) {};
  \node[tip] (x) at (1,1,2) {};
  \node[tip] (xp) at (3,3,0) {};
  \node[lab, above left=1pt and 2pt] at (x) {$x$};
  \node[lab, below right=1pt and 2pt] at (xp) {$x'$};
  \node[lab, below right=1pt and 2pt] at (1,1,0) {$\ell$};
  \node[lab, above right=1pt and 2pt] at (3,3,2) {$h$};
\end{tikzpicture}
\caption{The grid $G^{3}$ with $G=\{0,1,\dots,4\}$, as in the proof of
Proposition~\ref{prop:saturation} with $t=3$ and $L=4$, and the box
spanned by $x=(1,1,2)$ and $x'=(3,3,0)$. Its eight vertices are the
points of $\{1,3\}\times\{1,3\}\times\{0,2\}$, between the corners
$\ell=(1,1,0)$ and $h=(3,3,2)$. The sets $S(y)$ at these vertices form
the $3$-cube with base $S(\ell)$ and generators
$X_i=\{z_{i,q}:\ell_i<q\le h_i\}$, whose shape is $(r+2;2,2,2)$.}
\label{fig:box}
\end{figure}
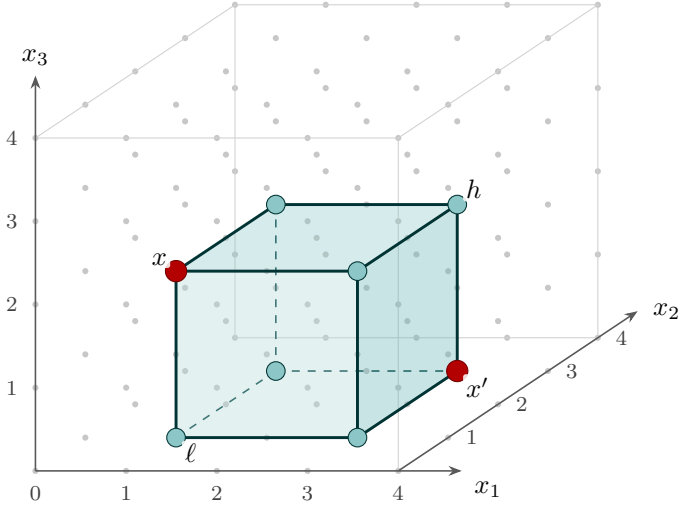

The proof is short enough that its two sources of loss are easy to
see.\footnote{The random ordering $\sigma$ is what turns a single good
shape into a positive proportion of all cubes of that shape. Without it
the box inequality would still produce a monochromatic cube in every
window, but it would say nothing about how many cubes of the same shape
are monochromatic, and it is this density that the averaging of
Section~\ref{sec:lower} consumes.} The factor $k^{1-2^{t}}$ is already the
probability that a fixed $t$-cube is monochromatic in a uniformly random
$k$-colouring, so for such colourings the expected value of $\rho_\chi$
at every shape equals $k^{1-2^{t}}$ and nothing better can be obtained on
average. The subtraction of $t/(L+1)$ only accounts for degenerate
boxes, and it is what forces the window $M$ to be as large as
$t^{2}k^{2^{t}}$.

\begin{remark}
\label{rem:hj}
Identify a subset $S$ of a finite set $Z$ with the word in $\{0,1\}^{Z}$ that has
the letter $1$ exactly at the positions in $S$. A $t$-cube then becomes a
$t$-dimensional combinatorial subspace: the positions in $U$ carry the
fixed letter $1$, those outside $T$ the fixed letter $0$, and those in
$X_i$ the $i$-th variable. Since $\rho_\chi>0$ in
Proposition~\ref{prop:saturation}, every $k$-colouring of the subsets of
a set of size $M_k(t)$ is constant on some $t$-cube, which is the
two-letter case of the multidimensional Hales--Jewett theorem
\cite{HJ63,GRS90} with the bound $t^{2}k^{2^{t}}$. The proposition says
more: a proportion at least $k^{-2^{t}}$ of the cubes of one shape are
monochromatic, and this density is what makes the rate in
Theorem~\ref{thm:lower} explicit.
\end{remark}

\section{Proof of Theorem~\ref{thm:lower}}
\label{sec:lower}

Fix $k\ge2$ and $t\ge1$, put $M=M_k(t)$ as in \eqref{eq:Mt}, and let $n$
be large enough that
\[
m=\Bigl\lfloor\frac{n-1}{M}\Bigr\rfloor\qquad\text{and}\qquad
s=\Bigl\lfloor\frac m2\Bigr\rfloor
\]
satisfy $s\ge t$. Note that $mM\le n-1$. Fix pairwise disjoint sets
$G_1,\dots,G_m\subseteq[n]$ of size $M$, the \emph{slots}, and list each
slot as $G_i=\{g_{i,1},\dots,g_{i,M}\}$; for instance
$g_{i,q}=(i-1)M+q$. For a vector $\beta\in[M]^{m}$ of block sizes let
\[
Y_i^{\beta}=\{g_{i,1},\dots,g_{i,\beta_i}\}\qquad(1\le i\le m),
\]
and for $\sigma\in S_n$ let $Q_{\sigma,\beta}(j,A)$ denote the cube
$Q(j,A)$ of Definition~\ref{def:chain} built from the blocks
$\sigma(Y_1^{\beta}),\dots,\sigma(Y_m^{\beta})$. Thus
$Q_{\sigma,\beta}(j,A)=\sigma\bigl(Q_{\mathrm{id},\beta}(j,A)\bigr)$, and
$Q_{\sigma,\beta}(j,A)$ has shape
\begin{equation}
\label{eq:chainshape}
\Bigl(\sum_{i\le j,\ i\notin A}\beta_i;\ (\beta_a)_{a\in A}\Bigr).
\end{equation}
Let $\cP$ be the set of pairs $(j,A)$ with $s<j\le m$, $A\subseteq[s]$
and $|A|=t$. Call $\beta$ \emph{good for} $(j,A)$ if the shape
\eqref{eq:chainshape} satisfies $\rho_\chi\ge k^{-2^{t}}$.

\begin{lemma}[Slices]
\label{lem:slices}
Let $1\le j\le m$ and $A\subseteq[j-1]$ with $|A|=t$. At least
$M^{m-t-1}$ vectors $\beta\in[M]^{m}$ are good for $(j,A)$.
\end{lemma}

\begin{proof}
Write $A=\{a_{(1)}<\dots<a_{(t)}\}$ and let $K=[m]\setminus(A\cup\{j\})$,
a set of $m-t-1$ indices. We show that every choice of $(\beta_i)_{i\in
K}\in[M]^{K}$ extends to a vector $\beta$ that is good for $(j,A)$.
Different choices give different vectors, which proves the lemma.

Fix $(\beta_i)_{i\in K}$, let $R=\sum_{i<j,\,i\notin A}\beta_i$, which
involves only indices in $K$, and put $r=R+1$. Since $R\le(j-1-t)M\le
(m-2)M$, we have $r+M\le(m-1)M+1\le n$. Proposition~\ref{prop:saturation}
gives $u$ and $a_1,\dots,a_t$ with $\rho_\chi(u;a)\ge k^{-2^{t}}$, where
$z=u-r$ satisfies $0\le z\le M-t\le M-1$ and $1\le a_i\le M$. Set
\[
\beta_j=z+1,\qquad \beta_{a_{(i)}}=a_i\quad(1\le i\le t).
\]
These values lie in $[M]$, and the shape \eqref{eq:chainshape} becomes
$(R+z+1;a_1,\dots,a_t)=(u;a_1,\dots,a_t)$. So $\beta$ is good for
$(j,A)$.
\end{proof}

\begin{lemma}[Averaging]
\label{lem:average}
There are $\sigma\in S_n$ and $\beta\in[M]^{m}$ such that at least
\[
\frac{|\cP|}{k^{2^{t}}M^{t+1}}
\]
of the cubes $Q_{\sigma,\beta}(j,A)$, $(j,A)\in\cP$, are monochromatic.
\end{lemma}

\begin{proof}
Count the triples $(\sigma,\beta,(j,A))\in S_n\times[M]^m\times\cP$ for
which $Q_{\sigma,\beta}(j,A)$ is monochromatic. For fixed $(j,A)\in\cP$
and $\beta$ good for $(j,A)$, the number of suitable $\sigma$ is $n!$
times the value of $\rho_\chi$ at the shape \eqref{eq:chainshape}, which
is at least $n!\,k^{-2^{t}}$. By Lemma~\ref{lem:slices} the number of
triples is at least $|\cP|\,M^{m-t-1}\,n!\,k^{-2^{t}}$. There are
$n!\,M^{m}$ pairs $(\sigma,\beta)$, so one of them occurs in at least
$|\cP|\,k^{-2^{t}}M^{-t-1}$ triples.
\end{proof}

\begin{proof}[Proof of Theorem~\ref{thm:lower}]
Let $c_{k,t}$ be as in the statement. Since $M\ge t+1$ we have
$(4M^{2})^{t+1}\ge(4M(t+1))^{t+1}$, so for $1\le n<4M(t+1)$ the bound
$F_k(n)\ge1>c_{k,t}n^{t+1}$ is trivial. Let $n\ge4M(t+1)$.

Take $\sigma,\beta$ as in Lemma~\ref{lem:average} and apply
Proposition~\ref{prop:glue} to the blocks
$\sigma(Y_1^\beta),\dots,\sigma(Y_m^\beta)$, which are pairwise disjoint
and nonempty, and to the set $\cS\subseteq\cP$ of pairs whose cube is
monochromatic. Every $k$-colouring therefore has a monochromatic
union-closed family with at least
\begin{equation}
\label{eq:Pbound}
\frac{|\cP|}{k^{2^{t}+1}M^{t+1}}
\end{equation}
members. It remains to bound $|\cP|=(m-s)\binom st$ from below. Since
$n-1<M(m+1)$ and $m\le2s+1$, we have $n\le2M(s+1)$. Together with
$n\ge4M(t+1)$ this gives $s\ge2t+1$, so $s\ge1$ and therefore $n\le4Ms$.
Also $m-s\ge s$, and $\binom st\ge(s/t)^{t}$ because $s\ge t$. Hence
\[
|\cP|\ \ge\ \frac{s^{t+1}}{t^{t}}\ \ge\
\frac{n^{t+1}}{t^{t}(4M)^{t+1}},
\]
and \eqref{eq:Pbound} is at least $c_{k,t}\,n^{t+1}$.

For the last assertion let $k=2$, so that $\log_2M=2\log_2t+2^{t}$ and
\begin{align*}
\log_2\frac1{c_{2,t}}&=2^{t}+1+t\log_2t+(t+1)(2+2\log_2M)\\
&=(2t+3)\,2^{t}+(5t+4)\log_2t+2t+3 .
\end{align*}
Since $\log_2t\le t-1$ and $2^{t}\ge2t$, the last two terms are at most
$5t^{2}+t-1\le2t(t+1)^{2}\le(t+1)^{2}2^{t}$, so
$\log_2(1/c_{2,t})\le(t+2)^{2}2^{t}$. If $\log_2n\ge(t+2)^{2}2^{t}$, then
$c_{2,t}n\ge1$ and $F(n)\ge c_{2,t}n^{t+1}\ge n^{t}$.
\end{proof}

\section{The upper bound and the random colouring}
\label{sec:upper}

A family $A_1,\dots,A_d$ of subsets of $[n]$ is an \emph{independent
system} if each $A_i$ has a \emph{private element}, that is, some
$x_i\in A_i$ lying in no $A_{i'}$ with $i'\ne i$. A set $X\subseteq[n]$ is
\emph{shattered} by a family $\cF$ if $\{A\cap X:\ A\in\cF\}=2^{X}$, and
the \emph{Vapnik--Chervonenkis dimension} $\vc(\cF)$ is the largest size of
a shattered set.

\begin{lemma}
\label{lem:indep}
Let $d\ge0$ and $\cF\subseteq2^{[n]}$.
\begin{enumerate}[label=\textup{(\alph*)},leftmargin=2em]
\item If $A_1,\dots,A_d$ is an independent system, the $2^{d}-1$ unions
$\bigcup_{i\in I}A_i$, $\emptyset\ne I\subseteq[d]$, are pairwise
distinct.
\item If $\cF$ shatters a set of size $d$, then $\cF$ contains an
independent system of $d$ sets.
\item \textup{(Sauer \cite{Sa72}, Shelah \cite{Sh72})} If $\vc(\cF)\le
d$, then $|\cF|\le\sum_{j=0}^{d}\binom nj$.
\end{enumerate}
\end{lemma}

\begin{proof}
(a) The private element $x_i$ lies in $\bigcup_{i'\in I}A_{i'}$ exactly
when $i\in I$, so the union determines $I$. (b) If $X=\{x_1,\dots,x_d\}$
is shattered, choose $A_i\in\cF$ with $A_i\cap X=\{x_i\}$; then $x_i$ is a
private element of $A_i$. (c) This is the Sauer--Shelah lemma.
\end{proof}

\begin{proof}[Proof of Theorem~\ref{thm:upper}]
Let $\chi$ be a uniformly random $2$-colouring of $2^{[n]}$, the colours
of the $2^{n}$ sets being independent and uniform. Call an ordered
$d$-tuple $(A_1,\dots,A_d)$ of subsets of $[n]$ \emph{bad} if it is an
independent system and the $2^{d}-1$ sets $\bigcup_{i\in I}A_i$,
$\emptyset\ne I\subseteq[d]$, all have the same colour. By
Lemma~\ref{lem:indep}(a) these sets are distinct, so a fixed independent
tuple is bad with probability $2^{1-(2^d-1)}=2^{2-2^{d}}$. There are at
most $2^{nd}$ tuples, and $2^{nd}\cdot2^{2-2^{d}}<1$ by hypothesis, so
some colouring has no bad tuple. Fix it.

Let $\cF$ be a monochromatic union-closed family. If $\cF$ contained an
independent system of $d$ sets, all unions of its members would lie in
$\cF$ and have the same colour, and the system would be bad. So $\cF$
contains no independent system of $d$ sets, $\vc(\cF)\le d-1$ by
Lemma~\ref{lem:indep}(b), and $|\cF|\le\sum_{j<d}\binom nj$ by
Lemma~\ref{lem:indep}(c).

For the second assertion let $n\ge2$, and put $a=\lfloor\log_2n\rfloor$
and $b=\lfloor\log_2a\rfloor$, so that $a\ge1$, $n<2^{a+1}$, $a<2^{b+1}$
and $b+1\le a$. For $d=a+b+3$ we get
\[
2^{d}=2^{a+1}\cdot2\cdot2^{b+1}\ \ge\ 2(n+1)(a+1),
\]
whereas $nd+2=n(a+1)+n(b+2)+2\le2n(a+1)+2<2(n+1)(a+1)$, because
$b+2\le a+1$ and $a\ge1$. So $2^{d}>nd+2$, and the first assertion gives
$F(n)\le\sum_{j<d}\binom nj\le\sum_{j<d}n^{j}\le d\,n^{d-1}$. Finally
$a\le\log_2n$ and $b\le\log_2a\le\log_2\log_2n$, so $d\le\log_2n+
\log_2\log_2n+3$.
\end{proof}

Figure~\ref{fig:surface} shows the size of the bound of
Theorem~\ref{thm:upper} as a function of $n$ and $d$. Taking for $d$ the
least admissible value $d(n)$ gives a bound valid for every $n$, and
$d(n)=\log_2n+\log_2\log_2n+O(1)$ as $n\to\infty$.

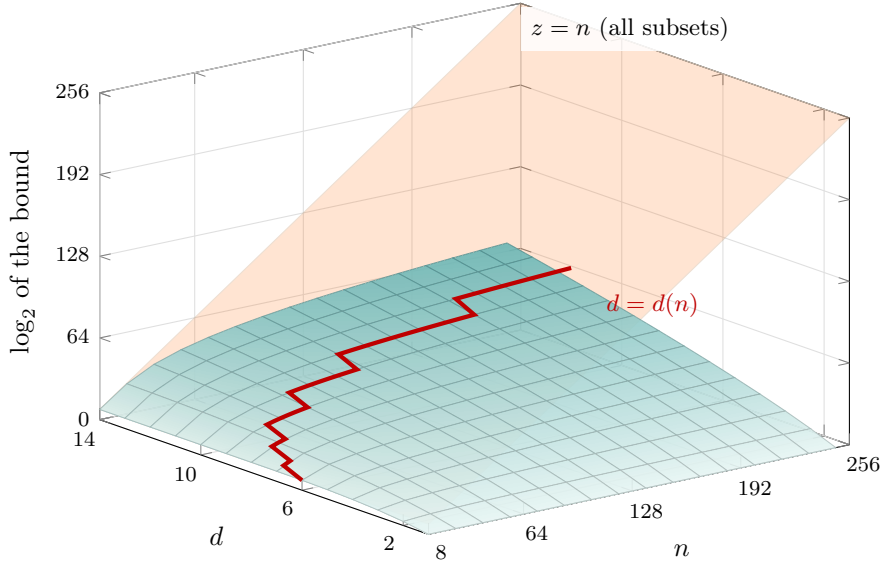
\begin{figure}[t]
\centering
\begin{tikzpicture}
  \begin{axis}[
      width=11.5cm, height=8.6cm,
      view={-38}{24},
      xlabel={$n$}, ylabel={$d$}, zlabel={$\log_2$ of the bound},
      xmin=8, xmax=256, ymin=1, ymax=14, zmin=0, zmax=256,
      xtick={8,64,128,192,256}, ytick={2,6,10,14}, ztick={0,64,128,192,256},
      tick label style={font=\scriptsize}, label style={font=\small},
      grid=major, grid style={gray!25},
      colormap={teal}{rgb255=(230,245,243) rgb255=(0,128,128) rgb255=(10,60,70)},
      z buffer=sort, clip=false]
    \addplot3[surf, shader=flat, opacity=0.18, fill=orange!80!red, draw=orange!60!black,
      samples=2, domain=8:256, y domain=1:14] {x};
    \addplot3[surf, shader=faceted interp, opacity=0.9, draw=black!30, line width=0.1pt] table {
8 1 0.000
24 1 0.000
40 1 0.000
56 1 0.000
72 1 0.000
88 1 0.000
104 1 0.000
120 1 0.000
136 1 0.000
152 1 0.000
168 1 0.000
184 1 0.000
200 1 0.000
216 1 0.000
232 1 0.000
248 1 0.000

8 2 3.170
24 2 4.644
40 2 5.358
56 2 5.833
72 2 6.190
88 2 6.476
104 2 6.714
120 2 6.919
136 2 7.098
152 2 7.257
168 2 7.401
184 2 7.531
200 2 7.651
216 2 7.762
232 2 7.864
248 2 7.960

8 3 5.209
24 3 8.234
40 3 9.681
56 3 10.641
72 3 11.360
88 3 11.936
104 3 12.415
120 3 12.826
136 3 13.186
152 3 13.505
168 3 13.793
184 3 14.055
200 3 14.295
216 3 14.517
232 3 14.722
248 3 14.914

8 4 6.539
24 4 11.183
40 4 13.385
56 4 14.839
72 4 15.926
88 4 16.794
104 4 17.517
120 4 18.136
136 4 18.678
152 4 19.159
168 4 19.592
184 4 19.986
200 4 20.347
216 4 20.680
232 4 20.989
248 4 21.278

8 5 7.349
24 5 13.661
40 5 16.639
56 5 18.597
72 5 20.057
88 5 21.222
104 5 22.190
120 5 23.019
136 5 23.744
152 5 24.388
168 5 24.968
184 5 25.494
200 5 25.976
216 5 26.422
232 5 26.835
248 5 27.220

8 6 7.775
24 6 15.759
40 6 19.536
56 6 22.008
72 6 23.846
88 6 25.311
104 6 26.528
120 6 27.569
136 6 28.478
152 6 29.286
168 6 30.012
184 6 30.672
200 6 31.277
216 6 31.835
232 6 32.352
248 6 32.835

8 7 7.948
24 7 17.536
40 7 22.133
56 7 25.129
72 7 27.352
88 7 29.120
104 7 30.588
120 7 31.843
136 7 32.939
152 7 33.911
168 7 34.786
184 7 35.580
200 7 36.307
216 7 36.978
232 7 37.600
248 7 38.181

8 8 7.994
24 8 19.032
40 8 24.470
56 8 28.001
72 8 30.615
88 8 32.691
104 8 34.412
120 8 35.882
136 8 37.166
152 8 38.304
168 8 39.328
184 8 40.257
200 8 41.107
216 8 41.892
232 8 42.620
248 8 43.299

8 9 8.000
24 9 20.278
40 9 26.578
56 9 30.654
72 9 33.664
88 9 36.051
104 9 38.029
120 9 39.717
136 9 41.189
152 9 42.495
168 9 43.668
184 9 44.733
200 9 45.708
216 9 46.607
232 9 47.441
248 9 48.218

8 10 8.000
24 10 21.298
40 10 28.477
56 10 33.109
72 10 36.523
88 10 39.225
104 10 41.461
120 10 43.369
136 10 45.032
152 10 46.506
168 10 47.830
184 10 49.031
200 10 50.131
216 10 51.145
232 10 52.085
248 10 52.962

8 11 8.000
24 11 22.114
40 11 30.186
56 11 35.385
72 11 39.208
88 11 42.229
104 11 44.728
120 11 46.857
136 11 48.712
152 11 50.356
168 11 51.831
184 11 53.170
200 11 54.395
216 11 55.524
232 11 56.571
248 11 57.547

8 12 8.000
24 12 22.746
40 12 31.718
56 12 37.496
72 12 41.734
88 12 45.079
104 12 47.841
120 12 50.194
136 12 52.243
152 12 54.058
168 12 55.686
184 12 57.163
200 12 58.514
216 12 59.759
232 12 60.913
248 12 61.989

8 13 8.000
24 13 23.216
40 13 33.086
56 13 39.452
72 13 44.113
88 13 47.785
104 13 50.815
120 13 53.393
136 13 55.638
152 13 57.624
168 13 59.406
184 13 61.022
200 13 62.499
216 13 63.861
232 13 65.123
248 13 66.299

8 14 8.000
24 14 23.545
40 14 34.300
56 14 41.265
72 14 46.354
88 14 50.358
104 14 53.658
120 14 56.464
136 14 58.905
152 14 61.065
168 14 63.002
184 14 64.757
200 14 66.362
216 14 67.841
232 14 69.211
248 14 70.488

    };
    \addplot3[line width=1.6pt, red!75!black, mark=none] coordinates {(8,6,7.775) (12,7,11.293) (16,7,13.862) (20,8,17.074) (24,8,19.032) (28,8,20.683) (32,9,23.842) (36,9,25.289) (40,9,26.578) (44,9,27.738) (48,9,28.793) (52,9,29.761) (56,9,30.654) (60,10,34.050) (64,10,34.927) (68,10,35.749) (72,10,36.523) (76,10,37.253) (80,10,37.944) (84,10,38.600) (88,10,39.225) (92,10,39.821) (96,10,40.391) (100,10,40.937) (104,11,44.728) (108,11,45.290) (112,11,45.831) (116,11,46.353) (120,11,46.857) (124,11,47.343) (128,11,47.814) (132,11,48.270) (136,11,48.712) (140,11,49.141) (144,11,49.557) (148,11,49.962) (152,11,50.356) (156,11,50.739) (160,11,51.112) (164,11,51.476) (168,11,51.831) (172,11,52.178) (176,11,52.516) (180,11,52.847) (184,11,53.170) (188,12,57.511) (192,12,57.852) (196,12,58.186) (200,12,58.514) (204,12,58.834) (208,12,59.148) (212,12,59.456) (216,12,59.759) (220,12,60.055) (224,12,60.346) (228,12,60.632) (232,12,60.913) (236,12,61.189) (240,12,61.460) (244,12,61.727) (248,12,61.989) (252,12,62.247) (256,12,62.501)};
    \node[font=\footnotesize, fill=white, fill opacity=0.85, text opacity=1, anchor=west]
      at (axis cs:256,14,236) {$z=n$ (all subsets)};
    \node[font=\footnotesize, text=red!75!black, anchor=west]
      at (axis cs:256,11,40) {$d=d(n)$};
  \end{axis}
\end{tikzpicture}
\caption{The surface $\log_2\sum_{j<d}\binom nj$ over $8\le n\le256$ and
$1\le d\le14$, with the plane $z=n$ of the trivial bound $2^{n}$. The red
curve follows the least admissible $d=d(n)$ of Theorem~\ref{thm:upper},
along which the bound stays far below the plane.}
\label{fig:surface}
\end{figure}

\begin{proof}[Proof of Corollary~\ref{cor:conj}]
(i) Let $n\ge2$, so that $\ell=\log_2\log_2n\ge0$. If $\omega(n)\le0$
there is nothing to prove, since $F(n)\ge1$. Otherwise
$x=\ell-2\log_2(\ell+2)>1$, and the integer $t=\lfloor x\rfloor$
satisfies $t\ge1$ and $\omega(n)=x-1<t\le x$. As $t\le\ell$ we have
$\log_2(t+2)\le\log_2(\ell+2)$, and so
\[
\log_2\bigl((t+2)^{2}2^{t}\bigr)=t+2\log_2(t+2)\ \le\ x+2\log_2(\ell+2)=\ell ,
\]
that is, $(t+2)^{2}2^{t}\le\log_2n$. Theorem~\ref{thm:lower} gives
$F(n)\ge n^{t}\ge n^{\omega(n)}$. As $\ell\to\infty$ and
$\log_2(\ell+2)=o(\ell)$, we have $\omega(n)\to\infty$.

(ii) Put $\eps(n)=F(n)^{1/n}-1+1/n$ for $n\ge1$. Then
$(1+\eps(n))^{n}=(F(n)^{1/n}+1/n)^{n}>F(n)$. By Theorem~\ref{thm:upper},
$\log F(n)=O((\log n)^{2})$, so $F(n)^{1/n}\to1$ and $\eps(n)\to0$.

(iii) Let $n\ge2$ and write $\lambda=\log_2n$ and
$D=\lambda+\log_2\lambda+3$. Theorem~\ref{thm:upper} gives
$n^{\omega(n)}\le F(n)\le D\,n^{D-1}$, hence
\[
\omega(n)\ \le\ D-1+\frac{\log_2D}{\log_2n}
\ =\ \lambda+O(\log\lambda),
\]
which is at most $(1+\eps)\lambda$ for all large $n$.

(iv) For every $t\ge1$, Theorem~\ref{thm:lower} gives $F_k(n)\ge
c_{k,t}n^{t+1}\ge n^{t}$ as soon as $n\ge1/c_{k,t}$. Let $\omega_k(n)$ be
the largest integer $t\le n$ with $n^{t}\le F_k(n)$; the value $t=0$
qualifies because $F_k(n)\ge1$. Then $\omega_k(n)\ge t$ whenever $n\ge
\max(t,1/c_{k,t})$, so $\omega_k(n)\to\infty$. For the second statement
take $\eps_k=\eps$ from (ii) and use $F_k(n)\le F(n)$.
\end{proof}

The random colouring does contain large monochromatic union-closed
families. How large they are indicates how much room is left between
Theorems~\ref{thm:lower} and~\ref{thm:upper}.

\begin{proposition}
\label{prop:random}
Let $n\ge1$. The average, over all $2$-colourings $\chi$ of $2^{[n]}$, of
the largest size of a monochromatic union-closed family is at least
\[
\sum_{j=0}^{n}\binom nj\,2^{1-2^{j}} .
\]
If $1\le j\le n$ and $2^{2^{j}}\le n$, this average is at least
$(2/n)(n/j)^{j}$.
\end{proposition}

\begin{proof}
Let $c=\chi([n])$ and let $\cY$ be the set of $Y\subseteq[n]$ such that
$\chi([n]\setminus Y')=c$ for every $Y'\subseteq Y$. Then $\cY$ is closed
under taking subsets and under intersection, so
$\cF=\{[n]\setminus Y:\ Y\in\cY\}$ is monochromatic and union-closed,
because $([n]\setminus Y)\cup([n]\setminus Y')=[n]\setminus(Y\cap Y')$.
A set $Y$ of size $j$ lies in $\cY$ exactly when the $2^{j}-1$ distinct
sets $[n]\setminus Y'$ with $\emptyset\ne Y'\subseteq Y$ all have colour
$c$, which happens for a proportion $2^{1-2^{j}}$ of all colourings.
Summing over $Y$ gives the average of $|\cF|$, which is the first bound.
For the second, the term of index $j$ is at least
$\binom nj2^{1-2^{j}}\ge(n/j)^{j}\cdot2/n$, since $\binom nj\ge(n/j)^{j}$
and $2^{2^{j}}\le n$.
\end{proof}

For $n\ge4$ and $j=\lfloor\log_2\log_2n\rfloor$ the conditions of
Proposition~\ref{prop:random} hold, and $(2/n)(n/j)^{j}\ge
n^{\log_2\log_2n-2-o(1)}$ because $j>\log_2\log_2n-1$ and
$j^{j}=n^{o(1)}$. So for the random
colouring the exponent $\omega$ is at least about $\log_2\log_2n$, while
Theorem~\ref{thm:upper}, which is also proved with the random colouring,
allows it to be as large as $\log_2n$. The next two statements turn the
upper bound for the random colouring into a counting problem.

\begin{lemma}
\label{lem:subfamily}
If $\cF\subseteq2^{[n]}$ is union-closed and $N\le|\cF|$, then $\cF$
contains a union-closed subfamily with exactly $N$ members.
\end{lemma}

\begin{proof}
It suffices to remove one member when $|\cF|>N$. Let $A$ be a member of
least size. If $B,C\in\cF\setminus\{A\}$ satisfied $B\cup C=A$, then
$B\subseteq A$ and $|B|\ge|A|$ would give $B=A$. So $\cF\setminus\{A\}$
is union-closed.
\end{proof}

\begin{proposition}
\label{prop:count}
Let $U(n,N)$ be the number of union-closed families $\cF\subseteq
2^{[n]}$ with $|\cF|=N$. If $U(n,N)<2^{N-1}$, then $F(n)<N$.
\end{proposition}

\begin{proof}
A fixed family of $N$ distinct sets is monochromatic under exactly
$2\cdot2^{2^{n}-N}$ of the $2^{2^{n}}$ colourings. The colourings under
which some union-closed family of $N$ sets is monochromatic are therefore
at most $U(n,N)\,2^{2^{n}+1-N}<2^{2^{n}}$ in number, and some colouring
has no monochromatic union-closed family of exactly $N$ members. By
Lemma~\ref{lem:subfamily} it has none with $N$ or more members, since a
subfamily of a monochromatic family is monochromatic. Hence
$F(n)\le N-1$.
\end{proof}

Proposition~\ref{prop:count} is the first-moment method for a uniformly
random colouring, and Corollary~\ref{cor:conj}(i) shows where it must
fail: if $N\le n^{\omega(n)}\le F(n)$, then $U(n,N)\ge2^{N-1}$. We return
to this in Question~\ref{q:count}.

\section{Cardinality colourings, Hilbert cubes and arithmetic progressions}
\label{sec:card}

In this section $k\ge2$ and $\chi(S)=g(|S|)$ for a function
$g:\{0,1,\dots,n\}\to[k]$. A $t$-cube of shape $(u;a_1,\dots,a_t)$
consists of sets whose sizes are the elements of $H(u;a_1,\dots,a_t)$, and
every permutation preserves sizes, so
\begin{equation}
\label{eq:rho01}
\rho_\chi(u;a_1,\dots,a_t)=
\begin{cases}1&\text{if $g$ is constant on }H(u;a_1,\dots,a_t),\\
0&\text{otherwise.}\end{cases}
\end{equation}
Figure~\ref{fig:hilbert} shows a monochromatic Hilbert cube of dimension
two.

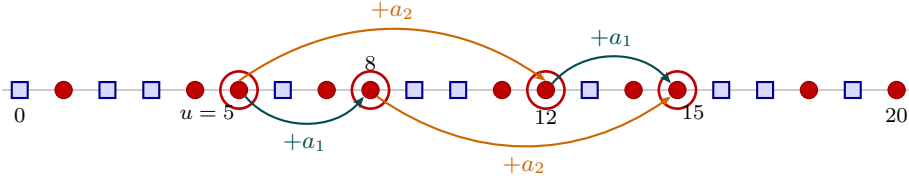
\begin{figure}[t]
\centering
\begin{tikzpicture}[x=0.58cm,
    r/.style={circle, draw=red!55!black, fill=red!75!black, inner sep=0pt, minimum size=6.5pt},
    b/.style={rectangle, draw=blue!55!black, fill=blue!15, line width=0.8pt, inner sep=0pt, minimum size=6pt},
    arc/.style={line width=0.8pt, -{Stealth[length=4.5pt]}},
    num/.style={font=\scriptsize, inner sep=1.5pt}]
  \draw[gray!60, line width=0.5pt] (-0.4,0) -- (20.4,0);
  \foreach \k/\c in {0/b,1/r,2/b,3/b,4/r,5/r,6/b,7/r,8/r,9/b,10/b,11/r,12/r,13/b,14/r,15/r,16/b,17/b,18/r,19/b,20/r} {
    \node[\c] (p\k) at (\k,0) {};
  }
  \foreach \k in {5,8,12,15} { \draw[red!75!black, line width=1pt] (\k,0) circle (7pt); }
  \draw[arc, orange!80!black] (p5.north) to[bend left=35]
    node[midway, above, font=\footnotesize] {$+a_2$} (p12.north);
  \draw[arc, teal!60!black] (p12.north east) to[bend left=50]
    node[midway, above, font=\footnotesize] {$+a_1$} (p15.north west);
  \draw[arc, teal!60!black] (p5.south east) to[bend right=50]
    node[midway, below, font=\footnotesize] {$+a_1$} (p8.south west);
  \draw[arc, orange!80!black] (p8.south east) to[bend right=35]
    node[midway, below, font=\footnotesize] {$+a_2$} (p15.south west);
  \node[num, below=5pt] at (0,0) {$0$};
  \node[num, below=5pt] at (20,0) {$20$};
  \node[num, below left=4pt and 0pt] at (5,0) {$u=5$};
  \node[num, above=6pt] at (8,0) {$8$};
  \node[num, below=6pt] at (12,0) {$12$};
  \node[num, below right=4pt and 0pt] at (15,0) {$15$};
  \node[font=\small, anchor=west] at (-0.4,2.0)
    {$H(u;a_1,a_2)=\{u,\ u+a_1,\ u+a_2,\ u+a_1+a_2\}$ with $u=5$, $a_1=3$, $a_2=7$};
\end{tikzpicture}
\caption{A colouring $g$ of $\{0,1,\dots,20\}$, red discs and blue
squares, and a red Hilbert cube $H(5;3,7)=\{5,8,12,15\}$. For the
cardinality colouring $\chi(S)=g(|S|)$, every $2$-cube of shape $(5;3,7)$
is red.}
\label{fig:hilbert}
\end{figure}

Since $\rho_\chi$ takes only the values $0$ and $1$ here,
Proposition~\ref{prop:saturation} turns into Hilbert's lemma with an
explicit window.

\begin{corollary}
\label{cor:hilbertlemma}
Let $k\ge2$, $t\ge1$, $M=t^{2}k^{2^{t}}$ and $r\ge0$. Every
$g:\{r,r+1,\dots,r+M\}\to[k]$ is constant on some Hilbert cube
$H(u;a_1,\dots,a_t)$ with $a_i\ge1$, $u\ge r$ and $u+a_1+\dots+a_t\le
r+M$.
\end{corollary}

\begin{proof}
Extend $g$ arbitrarily to $\{0,\dots,r-1\}$ and apply
Proposition~\ref{prop:saturation}, with $n=r+M$, $L=L_k(t)$ and
$M=M_k(t)$, to the cardinality colouring $S\mapsto g(|S|)$ of
$2^{[r+M]}$. It gives a shape with $r\le u$, $(u-r)+\sum_ia_i\le M$ and
$\rho_\chi(u;a)\ge k^{-2^{t}}>0$, so $\rho_\chi(u;a)=1$ and $g$ is
constant on $H(u;a)$ by \eqref{eq:rho01}. The elements of $H(u;a)$ lie
between $u\ge r$ and $u+\sum_ia_i\le r+M$.
\end{proof}

\begin{proof}[Proof of Theorem~\ref{thm:hilbert}]
Consider the cardinality colouring $\chi(S)=g(|S|)$ and the construction
of Section~\ref{sec:lower}, whose hypotheses hold since $n\ge4M(t+1)$
gives $s\ge n/(4M)\ge t+1$. By \eqref{eq:rho01}, a vector $\beta$ is good
for $(j,A)$ precisely when $g$ is constant on the Hilbert cube formed by
the sizes of the members of $Q_{\mathrm{id},\beta}(j,A)$, and these sizes
are
\[
|C_j\setminus Y_B|=\beta_1+\dots+\beta_j-\sum_{b\in B}\beta_b
\qquad(B\subseteq A),
\]
that is, the elements of $H_\beta(j,A)$. By Lemma~\ref{lem:slices}, for
every $(j,A)\in\cP$ at least $M^{m-t-1}$ vectors $\beta$ make
$H_\beta(j,A)$ monochromatic. Averaging over the $M^{m}$ vectors, as in
Lemma~\ref{lem:average} but without the permutations, some $\beta$ makes
at least $|\cP|/M^{t+1}$ of the cubes $H_\beta(j,A)$ monochromatic, and
at least a proportion $1/k$ of these share a colour $c$. The estimate
$|\cP|\ge n^{t+1}/(t^{t}(4M)^{t+1})$ from the proof of
Theorem~\ref{thm:lower} completes the proof. All elements of the cubes
lie in $\{0,\dots,n\}$, since $\beta_1+\dots+\beta_m\le mM\le n-1$.
\end{proof}

Given $\beta$ and $c$ as in Theorem~\ref{thm:hilbert}, the cubes
$Q_{\mathrm{id},\beta}(j,A)$ attached to the pairs of colour $c$ are
monochromatic under $\chi$, and Proposition~\ref{prop:glue} glues them
into a single union-closed family of colour $c$. For cardinality
colourings the proof of Theorem~\ref{thm:lower} is therefore entirely a
statement about sums of a bounded sequence of integers.

There is also a direct route through arithmetic progressions, which gives
Howorka's theorem.

\begin{proposition}
\label{prop:AP}
Suppose that $g$ is constant on the arithmetic progression
$a,a+e,\dots,a+(p-1)e$, where $e\ge1$, $p\ge1$ and $a+(p-1)e\le n$. Put
$q=\lfloor a/e\rfloor$ and $b=q+p-1$. Then $\chi$ has a monochromatic
union-closed family with $\sum_{i=0}^{p-1}\binom bi$ members and a
monochromatic sublattice with $2^{p-1}$ members.
\end{proposition}

\begin{proof}
Let $a_0=a-qe$, so $0\le a_0<e$, and let $D=\{1,\dots,a_0\}$ and
$X_i=\{a_0+(i-1)e+1,\dots,a_0+ie\}$ for $1\le i\le b$. These sets are
pairwise disjoint subsets of $[n]$, because $a_0+be=a+(p-1)e\le n$, and
$|D|=a_0$ and $|X_i|=e$. The family
\[
\cF=\bigl\{D\cup X_I:\ I\subseteq[b],\ |I|\ge q\bigr\}
\]
is union-closed, since $D\cup X_I\cup D\cup X_{I'}=D\cup X_{I\cup I'}$ and
$|I\cup I'|\ge q$ whenever $|I|\ge q$. A member with index set $I$ has
size $a_0+|I|e=a+(|I|-q)e$, where $0\le|I|-q\le p-1$, so all members have
the colour of the progression, and
\[
|\cF|=\sum_{i=q}^{b}\binom bi=\sum_{i=0}^{p-1}\binom bi
\]
by the symmetry $\binom bi=\binom b{b-i}$. For the sublattice take
$D'=\{1,\dots,a\}$ and $X'_i=\{a+(i-1)e+1,\dots,a+ie\}$ for $1\le i\le
p-1$. The $(p-1)$-cube $Q(D';X'_1,\dots,X'_{p-1})$ is a sublattice with
$2^{p-1}$ members, whose sizes $a+|I|e$ lie in the progression.
\end{proof}

For $k\ge1$ and $p\ge1$ let $W_k(p)$ be the least $W$ such that every
$g:\{0,1,\dots,W-1\}\to[k]$ is constant on some progression
$a,a+e,\dots,a+(p-1)e$ with $e\ge1$. This number exists by van der
Waerden's theorem.

\begin{theorem}[van der Waerden \cite{vdW27}]
\label{thm:vdw}
For all $k\ge1$ and $p\ge1$ there is $W$ such that every
$g:\{0,1,\dots,W-1\}\to[k]$ is constant on an arithmetic progression of
length $p$ with a positive common difference.
\end{theorem}

The formal proof derives Theorem~\ref{thm:vdw} from the Hales--Jewett
theorem.\footnote{Let $N$ be such that every $k$-colouring of the words
of length $N$ over the alphabet $\{0,1,\dots,p-1\}$ has a monochromatic
combinatorial line, and put $W=N(p-1)+1$. Colour a word by the value of
$g$ at the sum of its letters. A combinatorial line has a nonempty set
$E$ of active coordinates, which carry the letter $i$ at its $i$-th
point, and fixed letters elsewhere, with sum $a$ say. The sums of its
points are $a+i|E|$ for $0\le i\le p-1$, all less than $W$, so a
monochromatic line gives a monochromatic progression with common
difference $|E|$.}

\begin{theorem}[Howorka]
\label{thm:howorka}
Let $p\ge2$, let $W=W_k(p)$, and let $n\ge2W$. Every
cardinality $k$-colouring of $2^{[n]}$ has a monochromatic union-closed
family with at least
\[
\binom{\lfloor n/(2W)\rfloor+p-1}{p-1}\ \ge\
\frac{\bigl(\lfloor n/(2W)\rfloor+1\bigr)^{p-1}}{(p-1)!}
\]
members.
\end{theorem}

\begin{proof}
The integers $n-W,n-W+1,\dots,n-1$ lie in $\{0,\dots,n\}$, and by the
definition of $W$ the function $g$ is constant on a progression
$a,a+e,\dots,a+(p-1)e$ among them, with $e\ge1$. Here $a\ge n-W$ and
$a+(p-1)e\le n-1$, and $e\le(p-1)e\le W-1$ because $p\ge2$. Since $n\ge
2W$ we have $n-W\ge\lfloor n/2\rfloor$, and so
\[
\Bigl\lfloor\frac ae\Bigr\rfloor\ \ge\ \Bigl\lfloor\frac{n-W}{W}\Bigr\rfloor
\ \ge\ \Bigl\lfloor\frac{\lfloor n/2\rfloor}{W}\Bigr\rfloor
=\Bigl\lfloor\frac n{2W}\Bigr\rfloor .
\]
Proposition~\ref{prop:AP} gives a monochromatic union-closed family with
at least $\binom b{p-1}$ members, where $b=\lfloor a/e\rfloor+p-1\ge
\lfloor n/(2W)\rfloor+p-1$, and $\binom b{p-1}$ increases with $b$.
Finally, with $N=\lfloor n/(2W)\rfloor$,
\[
\binom{N+p-1}{p-1}=\frac{(N+1)(N+2)\cdots(N+p-1)}{(p-1)!}\ \ge\
\frac{(N+1)^{p-1}}{(p-1)!}. \qedhere
\]
\end{proof}

For fixed $p$ the bound of Theorem~\ref{thm:howorka} is at least a
constant multiple of $n^{p-1}$, and $p$ is arbitrary. So for cardinality
colourings the largest monochromatic union-closed family grows faster
than every fixed power of $n$, which is Howorka's theorem as recorded in
\cite{Bl26}; for all colourings this is contained in
Corollary~\ref{cor:conj}. Proposition~\ref{prop:AP} also restricts the
cardinality colourings that are good for $f$: if no monochromatic
sublattice has more than $m$ members, then $g$ is constant on no
progression of length $p$ in $\{0,\dots,n\}$ with $2^{p-1}>m$.

\section{The lattice function}
\label{sec:f}

The lower bound in Theorem~\ref{thm:f} is \eqref{eq:trivial}. The upper
bounds come from a random colouring and a count of sublattices, which in
turn rests on a description of a sublattice by a small amount of data.

Let $\cL\subseteq2^{[n]}$ be a nonempty sublattice, and put
$B=\bigcap_{A\in\cL}A$ and $T=\bigcup_{A\in\cL}A$, the least and the
greatest member of $\cL$. For $x\in T\setminus B$ let
\[
A_x=\bigcap\{A\in\cL:\ x\in A\},
\]
the least member of $\cL$ containing $x$, and for $X\subseteq
T\setminus B$ put $A_X=\bigcup_{x\in X}A_x$. For $x\in T\setminus B$ let
$[x]=\{y\in T\setminus B:\ A_y=A_x\}$ be the \emph{class} of $x$, and let
$D_x=\{y\in T\setminus B:\ A_y\subsetneq A_x\}$.

\begin{lemma}
\label{lem:structure}
Let $\cL$ be a nonempty sublattice of $2^{[n]}$.
\begin{enumerate}[label=\textup{(\alph*)},leftmargin=2em]
\item Every $A\in\cL$ equals $B\cup A_{A\setminus B}$, and $B\cup
A_X\in\cL$ for every $X\subseteq T\setminus B$. Hence
$\cL=\{B\cup A_X:\ X\subseteq T\setminus B\}$.
\item For $x,y\in T\setminus B$ we have $y\in A_x$ if and only if
$A_y\subseteq A_x$.
\item For $x\in T\setminus B$ we have $A_x=B\cup[x]\cup A_{D_x}$.
\end{enumerate}
\end{lemma}

\begin{proof}
(a) Let $A\in\cL$. Then $B\subseteq A\subseteq T$, and $A_x\subseteq A$
for every $x\in A\setminus B$, so $B\cup A_{A\setminus B}\subseteq A$;
conversely every $x\in A$ lies in $B$ or in $A_x$. The set $B\cup A_X$ is
a union of members of $\cL$, or equals $B$ when $X=\emptyset$.

(b) If $y\in A_x$, then $A_x$ is a member containing $y$, so $A_y\subseteq
A_x$. Conversely $y\in A_y$.

(c) The right-hand side is contained in $A_x$: we have $B\subseteq A_x$,
every $y\in[x]$ lies in $A_y=A_x$, and $A_y\subseteq A_x$ for $y\in D_x$.
Let $z\in A_x\setminus B$. Then $z\in T\setminus B$ and $A_z\subseteq A_x$
by (b), so either $A_z=A_x$ and $z\in[x]$, or $z\in D_x$ and $z\in
A_z\subseteq A_{D_x}$.
\end{proof}

Part (a) is the form that Birkhoff's representation theorem for finite
distributive lattices \cite{Bi37} takes for sublattices of
$2^{[n]}$.\footnote{The sets $A_x$ are exactly the members of $\cL$ other
than $B$ that are not the union of two smaller members: $A_x$ is not such
a union because each of the two parts would have to contain $x$, and by
(a) any other member is the union of the sets $A_x$ it contains, all of
them smaller. Birkhoff's theorem identifies a finite distributive lattice
with the down-sets of the poset of these elements.} Part (a) already
shows that $\cL$ is determined by $B$ and the map $x\mapsto A_x$. To count
large sublattices more efficiently we describe each $A_x$ through a small
set of points below it.

\begin{lemma}
\label{lem:cover}
Let $\cL$ be a nonempty sublattice and $x\in T\setminus B$, and let
$R_x\subseteq D_x$ be a set of least size with $A_{R_x}=A_{D_x}$. Then
$A_x=B\cup[x]\cup A_{R_x}$, no set $A_y$ with $y\in R_x$ is contained in
$A_{y'}$ for another $y'\in R_x$, and the $2^{|R_x|}$ sets $B\cup A_Y$,
$Y\subseteq R_x$, are distinct members of $\cL$. In particular
$|R_x|\le\log_2|\cL|$.
\end{lemma}

\begin{proof}
The first statement is Lemma~\ref{lem:structure}(c). If $A_y\subseteq
A_{y'}$ for distinct $y,y'\in R_x$, then $R_x\setminus\{y\}$ would be a
smaller set with the same union $A_{R_x}$. The sets $B\cup A_Y$ lie in
$\cL$ by Lemma~\ref{lem:structure}(a). Let $Y\ne Y'$ be subsets of $R_x$,
say $y\in Y\setminus Y'$. Then $y\in A_y\subseteq B\cup A_Y$. If $y$ lay
in $B\cup A_{Y'}$, then, as $y\notin B$, we would have $y\in A_{y'}$ for
some $y'\in Y'$, so $A_y\subseteq A_{y'}$ by
Lemma~\ref{lem:structure}(b) with $y'\ne y$, which is impossible. So
$B\cup A_Y\ne B\cup A_{Y'}$.
\end{proof}

For each nonempty sublattice fix sets $R_x$ as in Lemma~\ref{lem:cover},
and for $x\in T\setminus B$ let $r(x)$ be the least element of $[x]$.

\begin{lemma}
\label{lem:code}
Let $\cL$ and $\cL'$ be nonempty sublattices of $2^{[n]}$ with the same
least member $B$, the same greatest member $T$, the same map $r$ on
$T\setminus B$ and the same sets $R_x$ for $x\in T\setminus B$. Then
$\cL=\cL'$.
\end{lemma}

\begin{proof}
Write $A_x$ and $A'_x$ for the least members containing $x$ in $\cL$ and
$\cL'$. The class of $x$ is $\{y\in T\setminus B:\ r(y)=r(x)\}$ in both
lattices. We show $A_x=A'_x$ for $x\in T\setminus B$ by induction on
$|A_x|$. By Lemma~\ref{lem:cover},
\[
A_x=B\cup[x]\cup A_{R_x}\qquad\text{and}\qquad
A'_x=B\cup[x]\cup A'_{R_x}.
\]
Every $y\in R_x$ satisfies $A_y\subsetneq A_x$, so $|A_y|<|A_x|$ and
$A_y=A'_y$ by induction. Hence $A_x=A'_x$, and
Lemma~\ref{lem:structure}(a) gives $\cL=\cL'$.
\end{proof}

\begin{lemma}
\label{lem:count}
\begin{enumerate}[label=\textup{(\alph*)},leftmargin=2em]
\item The number of nonempty sublattices of $2^{[n]}$ is at most
$2^{n(n+1)}$.
\item For $w\ge0$, the number of nonempty sublattices $\cL$ of $2^{[n]}$
with $\lfloor\log_2|\cL|\rfloor=w$ is at most $2^{2n}\,n^{n}\,(n+1)^{wn}$.
\end{enumerate}
\end{lemma}

\begin{proof}
(a) By Lemma~\ref{lem:structure}(a), $\cL$ is determined by $B$ and the
map $x\mapsto A_x$, extended by $A_x=B$ for $x\notin T\setminus B$. There
are $2^{n}$ choices for $B$ and $(2^{n})^{n}$ for the map.

(b) By Lemma~\ref{lem:code}, $\cL$ is determined by $B$, $T$, the map
$r$ and the sets $R_x$, extended in any fixed way to all $x\in[n]$. By
Lemma~\ref{lem:cover} each $R_x$ has at most $w$ elements. There are at
most $2^{n}$ choices for each of $B$ and $T$, at most $n^{n}$ for $r$,
and, since the number of subsets of $[n]$ with at most $w$ elements is
$\sum_{i\le w}\binom ni\le\sum_{i\le w}n^{i}\le(n+1)^{w}$, at most
$(n+1)^{wn}$ for the family $(R_x)_{x\in[n]}$.
\end{proof}

\begin{proof}[Proof of Theorem~\ref{thm:f}]
The number of $2$-colourings under which a fixed family of $N$ distinct
sets is monochromatic is $2\cdot2^{2^{n}-N}$.

For the first bound put $m_1=n^{2}+n+2$. By Lemma~\ref{lem:count}(a),
the colourings under which some sublattice with at least $m_1$ members is
monochromatic number at most $2^{n(n+1)}\cdot2\cdot2^{2^{n}-m_1}=
2^{2^{n}-1}$. So some colouring has no such sublattice, and $f(n)\le
m_1-1=n^{2}+n+1$.

For the second bound let $L=\lfloor\log_2n\rfloor+1$, so that
$n+1\le2^{L}$, and for $w\ge0$ put
\[
E(w)=n\bigl(L(w+1)+2\bigr)+w+3 .
\]
The function $E$ is increasing, and $E(w)\le2^{w}$ for $w\ge3L+2$.
Indeed, with $P=2^{L}$ we have $n\le P-1$ and $L+1\le P$, so
\[
E(3L+2)=n(3L^{2}+3L+2)+3L+5\le P(3L^{2}+3L+5)\le4P^{3}=2^{3L+2},
\]
and $E(w+1)=E(w)+nL+1\le2E(w)$. Put $m_2=E(3L+1)=nL(3L+2)+2n+3L+4$.
Every sublattice $\cL$ with $|\cL|\ge m_2$ and $\lfloor\log_2|\cL|\rfloor
=w$ satisfies $|\cL|\ge E(w)$: for $w\le3L+1$ because $E(w)\le m_2$, and
for $w\ge3L+2$ because $E(w)\le2^{w}\le|\cL|$. By
Lemma~\ref{lem:count}(b), and using $n\le2^{L}$ and $n+1\le2^{L}$, there
are at most $2^{2n}n^{n}(n+1)^{wn}\le2^{n(L(w+1)+2)}$ such sublattices for
each $w$, and each is monochromatic under at most $2^{2^{n}+1-E(w)}$
colourings. The colourings under which some sublattice with at least
$m_2$ members is monochromatic therefore number at most
\[
\sum_{w=0}^{n}2^{n(L(w+1)+2)}\,2^{2^{n}+1-E(w)}
=2^{2^{n}}\sum_{w=0}^{n}2^{-(w+2)}\ <\ 2^{2^{n}-1},
\]
the sum running up to $w=n$ because a sublattice has at most $2^{n}$
members. So some colouring has no such sublattice, and $f(n)\le
m_2-1=nL(3L+2)+2n+3L+3$.

As $L\le\log_2n+1$, the second bound is $(3+o(1))\,n(\log_2n)^{2}$. The
last assertion of the theorem follows from Theorem~\ref{thm:lower} with
$t>C$.
\end{proof}

The first bound is the smaller one for $n<210$ and the second for
$n\ge263$. Both bounds come from the random colouring, and the gap
between them and the lower bound $\lceil(n+1)/2\rceil$ is the subject of
Question~\ref{q:f}.

\section{Formal verification}
\label{sec:lean}

The results of Sections~\ref{sec:cubes}--\ref{sec:f} have been
formalised in Lean~4 \cite{MU21}, version~4.34.1, with the Mathlib
library \cite{Mathlib20} at the release of the same version number. The
development contains no unproved steps, and for each statement listed
below Lean reports that the proof uses only the three standard axioms of
its logic, namely propositional extensionality, the axiom of choice and
the soundness of quotient types.\footnote{These axioms are part of the
foundations on which Mathlib is built. In particular, no statement relies
on an additional axiom or on evaluating code in place of a proof.}

In the formalisation the ground set is the type of $n$ elements, a
$k$-colouring is a function from its subsets to a finite type with $k$
elements, and $F_k(n)$, $F(n)$ and $f(n)$ are defined literally as the
minimum over all colourings of the largest size of a monochromatic
family that is closed under union, respectively under union and
intersection. It is also proved that $F_k(n)$ is the largest $m$ such that
every $k$-colouring has a monochromatic union-closed family of at least
$m$ sets. The following statements are checked.

\begin{enumerate}[label=(\arabic*),leftmargin=2em]
\item The bound \eqref{eq:trivial}, and $F_k(n)\le F(n)$ for $k\ge2$.
\item Theorem~\ref{thm:lower} in the stated form
$n^{t+1}\le k^{2^{t}+1}t^{t}(4M^{2})^{t+1}F_k(n)$ for $n\ge1$, and the
bound $F(n)\ge n^{t}$ for $n\ge2^{(t+2)^{2}2^{t}}$, together with
Lemmas~\ref{lem:shape}--\ref{lem:average} and
Propositions~\ref{prop:glue} and~\ref{prop:saturation} in the form in
which the proof uses them.
\item Lemma~\ref{lem:indep} and Theorem~\ref{thm:upper}, both the bound
$F(n)\le\sum_{j<d}\binom nj$ and the two explicit forms, with the
Sauer--Shelah lemma taken from Mathlib.
\item Corollary~\ref{cor:conj}: part~(i) with the stated $\omega(n)$ and
$\omega(n)\to\infty$, parts (ii)--(iv), and the conjunction of the two
conjectures in the form \eqref{eq:conj} as a single theorem.
\item Proposition~\ref{prop:random}, Lemma~\ref{lem:subfamily} and
Proposition~\ref{prop:count}.
\item Corollary~\ref{cor:hilbertlemma}, Theorem~\ref{thm:hilbert},
Proposition~\ref{prop:AP}, Theorem~\ref{thm:vdw}, derived from Mathlib's
Hales--Jewett theorem as described above, and Theorem~\ref{thm:howorka},
proved for every number $W$ with the property that defines $W_k(p)$.
\item Lemmas~\ref{lem:structure}--\ref{lem:count} and
Theorem~\ref{thm:f}.
\end{enumerate}

The formal proofs follow the arguments given here, with one change of
language: probabilities are replaced by counts. Permutations and size
vectors are counted rather than sampled, so that Lemma~\ref{lem:average}
becomes a double counting over the product of the two sets, and the
random colourings of Sections~\ref{sec:upper} and~\ref{sec:f} are
replaced by counts of colourings. The box inequality is proved for sums
over all pairs of points, which is Lemma~\ref{lem:box} multiplied by
$|G|^{2t}$. Outside the formal development remain only the figures and
the elementary asymptotic rewritings of the proved bounds, such as
$\omega(n)=(1-o(1))\log_2\log_2n$ in Corollary~\ref{cor:conj} and
$n^{\log_2\log_2n-2-o(1)}$ after Proposition~\ref{prop:random}.

\section{Questions}
\label{sec:questions}

With both conjectures settled, what remains open is the precise size of
the two functions, above all of $F$.

\begin{question}
\label{q:F}
What is the order of magnitude of $\log F(n)/\log n$? By
Corollary~\ref{cor:conj} it lies between $(1-o(1))\log_2\log_2n$ and
$(1+o(1))\log_2n$. Is $F(n)=n^{\Theta(\log\log n)}$, or is
$F(n)=n^{\Theta(\log n)}$?
\end{question}

The random colouring is the natural candidate for an upper bound, and by
Proposition~\ref{prop:random} it cannot give anything below
$n^{(1-o(1))\log_2\log_2n}$. Proposition~\ref{prop:count} reduces the
matching upper bound to a count.

\begin{question}
\label{q:count}
Is there a constant $C$ such that $U(n,N)<2^{N-1}$ for all large $n$ and
all $N\ge n^{C\log_2\log_2n}$?
\end{question}

A positive answer would give $F(n)<n^{C\log_2\log_2n}$ by
Proposition~\ref{prop:count}, and together with
Corollary~\ref{cor:conj}(i) it would determine $\log F(n)/\log n$ up to a
constant factor, answering Question~\ref{q:F} in favour of
$n^{\Theta(\log\log n)}$. By Corollary~\ref{cor:conj}(i) the inequality
fails for $N\le n^{\omega(n)}$, so the threshold in Question~\ref{q:count}
cannot be lowered below $n^{(1-o(1))\log_2\log_2n}$.

\begin{question}
\label{q:f}
Is $f(n)=O(n)$? Is $f(n)=\lceil(n+1)/2\rceil$ for every $n$?
\end{question}

The lower bound $\lceil(n+1)/2\rceil$ comes from a single chain, and the
upper bounds of Theorem~\ref{thm:f} from counting sublattices in a
random colouring, which leaves a factor of order $(\log n)^{2}$ between
them.

\section*{Data availability}

The Lean formalisation of the results of this paper and the sources of
the paper are available at
\href{https://doi.org/\zenododoi}{https://doi.org/\zenododoi}. Claude
(Anthropic) assisted with the coding.


\end{document}